\documentclass[12pt]{article}
\usepackage[a4paper, left=3.2cm, right=3.2cm, top=3.5cm, bottom=3.5cm]{geometry}

\usepackage[T1]{fontenc}
\usepackage[utf8]{inputenc}

\usepackage[T1]{fontenc}
\usepackage[helvratio=0.92]{newtxtext}

\usepackage{amsmath}
\usepackage{amsthm}
\usepackage{amsfonts}
\usepackage{amssymb}
\usepackage{graphicx}
\numberwithin{equation}{section}
\usepackage{color}
\usepackage{caption}
\usepackage{pdfrender,xcolor}
\usepackage{enumitem}
\usepackage{amssymb,amsmath,amsthm,enumitem}
\usepackage{enumitem}
\usepackage[arrow, matrix, curve]{xy}
\usepackage{xcolor}

        \usepackage[bookmarks=false]{hyperref}

\usepackage[T1]{fontenc}
\usepackage[utf8]{inputenc}

\DeclareRobustCommand{\rchi}{{\mathpalette\irchi\relax}}
\newcommand{\irchi}[2]{\raisebox{\depth}{$#1\chi$}} 

\usepackage{setspace}
\usepackage{epstopdf}
\usepackage[caption=false]{subfig}

\usepackage{color}
\DeclareMathOperator*{\esssup}{ess\,sup}
\makeatletter
\def\NAT@def@citea{\def\@citea{\NAT@separator}}
\makeatother

\usepackage[T1]{fontenc}
\usepackage[helvratio=0.92]{newtxtext}
\usepackage{newtxmath}
\usepackage{microtype}
\usepackage{authblk}
\usepackage{fancyhdr}

\emergencystretch=\maxdimen
\theoremstyle{plain}
\newtheorem{theorem}{Theorem}[section]
\newtheorem{lemma}[theorem]{Lemma}
\newtheorem{corollary}[theorem]{Corollary}
\newtheorem{proposition}[theorem]{Proposition}

\theoremstyle{definition}
\newtheorem{definition}[theorem]{Definition}

\theoremstyle{remark}
\newtheorem{remark}{Remark}

\newcommand{\R}{\mathbb{R}}
\newcommand{\Z}{\mathbb{Z}}
\newcommand{\N}{\mathbb{N}}
\newcommand{\Oe}{\Omega_\epsilon}

\usepackage{setspace}

\allowdisplaybreaks
\usepackage{titling}

\usepackage{blindtext}

\newcommand\shorttitle{Homogenization of nonlinear drift-diffusion in porous medium}
\newcommand\authors{A. Bhattacharya, M. Gahn and M. Neuss-Radu}

\usepackage{amsmath,amssymb,amsfonts,amscd}

\begin{document}
	
\title{\Large Homogenization of a nonlinear drift-diffusion system for multiple charged species in a porous medium}

\author{Apratim Bhattacharya$^{*}$, Markus Gahn$^{**}$ and Maria Neuss-Radu$^{*}$ 
}

\date{}
\maketitle

\begin{abstract}
We consider a nonlinear drift-diffusion system for multiple charged species in a porous medium in 2D and 3D with periodic microstructure. The system consists of a transport equation for the concentration of the species and Poisson's equation for the electric potential. The diffusion terms depend nonlinearly on the concentrations. We consider non-homogeneous Neumann boundary condition for the electric potential. The aim is the rigorous derivation of an effective (homogenized) model in the limit when the scale parameter $\epsilon$ tends to zero. This is based on uniform \textit{a priori} estimates for the solutions of the microscopic model. The crucial result is the uniform $L^\infty$-estimate for the concentration in space and time. This result exploits the fact that the system admits a nonnegative energy functional which decreases in time along the solutions of the system. By using weak and strong (two-scale) convergence properties of the microscopic solutions, effective models are derived in the limit $\epsilon \to 0$ for different scalings of the microscopic model.
\end{abstract}



\let\thefootnote\relax\footnotetext{$^{*}$Department Mathematik, Friedrich-Alexander-Universität Erlangen-Nürnberg, Cauerstr. 11, 91058 Erlangen, Germany, apratim.bhattacharya@fau.de, maria.neuss-radu@math.fau.de
	
	\vspace{.5mm}
	 \ $^{**}$Interdisciplinary Center for Scientific Computing, University of Heidelberg, Im Neuenheimer Feld
	205, 69120 Heidelberg, Germany, markus.gahn@iwr.uni-heidelberg.de}



\section{Introduction}

The aim of this paper is the rigorous homogenization of the nonlinear drift-diffusion model (non-dimensional) \eqref{non_dim_PNP}-\eqref{def_h_p} for a number of  $P \in \N$ charged species with concentrations $c_{i, \epsilon}, 1 \leq i \leq P$, and the electric potential $\phi_\epsilon$ in a periodically perforated domain $\Omega_\epsilon$ representing the fluid (pore) phase of a porous medium (see also Figure \ref{fig_domain_PNP}):
\begin{subequations}\label{non_dim_PNP}
\begin{align}
\partial_{t} c_{i,\epsilon}(t,x) +\nabla \cdot J_{i,\epsilon} (t,x) &=0  && \mbox{ in }  (0,T) \times \Omega_\epsilon, \label{non_dim_NP_eq}
\\
 J_{i,\epsilon} (t,x) \cdot \nu_\epsilon &=0   && \mbox{ on }  (0,T) \times \partial \Omega_\epsilon, \label{non_dim_bc_1}
 \\
 c_{i,\epsilon}(0,x) &= c^{0}_{i} (x)  && \mbox{ in } \Oe. \label{non_dim_ic}
\\
  -\epsilon^\alpha   \Delta \phi_\epsilon (t,x) &=  \sum_{i=1}^{P}  z_i c_{i,\epsilon}  (t,x)   && \mbox{ in }  (0,T) \times \Omega_\epsilon, \label{non_dim_poisson_eq}
\\
 \epsilon^{\alpha} \nabla \phi_\epsilon (t,x) \cdot \nu_\epsilon &  = \xi_\epsilon (x)    && \mbox{ on } (0,T) \times \partial \Omega_\epsilon \label{non_dim_bc_2},
 \end{align}
 \end{subequations}
where the total flux $J_{i,\epsilon}$ of the $i$-th charged species is given by
\begin{equation}\label{total_flux}
J_{i,\epsilon}(t, x) = - \left(  D_{i} \nabla h_p(c_{i,\epsilon})+ \epsilon^\beta  D_{i}z_i c_{i,\epsilon} (t,x) \nabla \phi_\epsilon (t,x) \right).
\end{equation}
Here $(0,T)$ is a time interval, whereas $D_i >0$ and $z_i \in \Z$ denote the (scaled) diffusivity and charge number of the $i$-th species, respectively. The (scaled) permittivity of the medium and the (scaled) mobility of the $i$-th charged species are given by $\epsilon^\alpha$ and $D_i \epsilon^\beta$, respectively, where $\alpha, \beta \in \mathbb{R}$ with $\alpha \leq \beta$. The scale parameter $\epsilon >0$ describes the length of the period of the porous microstructure and it  is also proportional to the radius of the pores. $\nu_\epsilon$ represents the outward unit normal vector to the boundary $\partial \Omega_\epsilon$. The function $h_p$ is defined by
\begin{equation}\label{def_h_p}
h_p(r)= r+ \eta r^p, \ \ r \geq 0, \ \eta \in (0,\infty), \  p \in [4, \infty).
\end{equation}
Equation (\ref{non_dim_NP_eq}) models the transport of the charged species due to nonlinear diffusion and electromigration in a domain with an impermeable boundary modeled by the no-flux boundary condition (\ref{non_dim_bc_1}) and with initial concentrations  given in (\ref{non_dim_ic}). The electric potential is induced by the charges of the species and is given as the solution  of Poisson's equation (\ref{non_dim_poisson_eq}) subject to the  non-homogeneous Neumann boundary condition (\ref{non_dim_bc_2}). The right hand side in \eqref{non_dim_poisson_eq} represents the (scaled) charge density within the fluid phase of the porous medium, while (\ref{non_dim_bc_2}) models a charged boundary and $\xi_\epsilon$ represents the (scaled) surface charge density. The parameters $\alpha, \beta \in \mathbb{R}$ (obtained from a non-dimensionalization procedure, see, e.g., \cite{RayThesis}, Section 2.1.3) allow to consider different scalings in our microscopic model corresponding to different settings and applications. 
\begin{figure}%
	\centering
	{{\includegraphics[width=5.5cm]{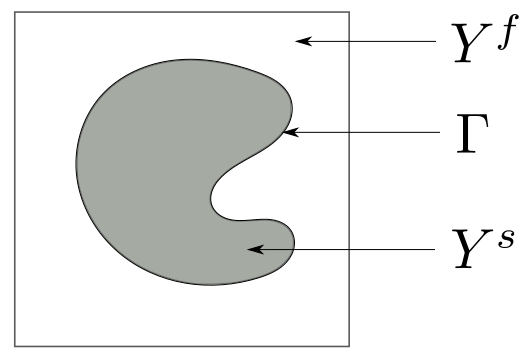} }}%
	\qquad 
	{{\includegraphics[width=6cm]{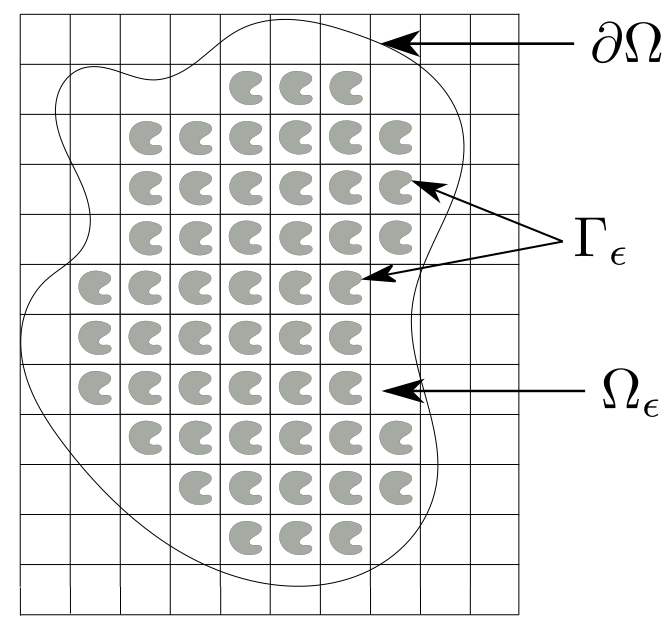} }}%
	\caption{The standard cell $Y= Y^f \cup \overline{Y^s}$ (left) and the porous medium $\Omega$ with the fluid (pore) part $\Omega_\epsilon$  (right). \label{fig_domain_PNP}}
\end{figure}

Drift-diffusion systems of the form \eqref{non_dim_PNP}-\eqref{total_flux} arise, e.g., in the mathematical modeling of semiconductors and of transport of charged particles (ions) in solutions (electrophoresis). In semiconductor modeling, see, e.g., \cite{Gaj86, Gli94, Mar90}, the case of two species is relevant (the electrons with valence $-1$ and the holes with valence $+1$) while in electrophoresis multiple types of charged particles (like, e.g., ions) with different charge numbers are transported in solution under the influence of an electric field, see, e.g., \cite{Dey79, Keen98}. If we consider, e.g., the case of ion channels located within the membrane of cells and intracellular organelles, it is known that there is a permanent charge on the atoms of the channel protein which can be measured by, for example, x-ray crystallography, see, e.g., \cite{Chen93}. This permanent charge which in our model is described by the surface charge density $\xi_\epsilon$ has a significant role in determining channels' permeation properties. In these applications, the function $h_p$ in \eqref{total_flux} can be linear or nonlinear. The linear case represents the classical drift-diffusion  (Poisson-Nernst-Planck) system. For a discussion about different (nonlinear) shapes of $h_p$, see, e.g., the introduction in \cite{Jun97}. The particular model with $h_p$ given by  \eqref{def_h_p} plays an important role in the approximation of the models used in applications. For example, in \cite{Both14} (inspired by contributions from  \cite{Gaj96}) this model is used as a regularization of the classical PNP system in order to show existence of the latter for multiple species in space dimension three and higher. Furthermore, in \cite{Jun97}, a similar model is used to approximate a drift-diffusion system with a (possible) degenerate nonlinear diffusion term by a nondegenerate system. 

In many applications, e.g., from geosciences, biology or biomedicine, electrophoretic processes take place in porous media. Due to the complex microstructure of the medium numerical simulations of microscopic models at the fine (pore) scale are very expensive. Therefore, effective (homogenized) approximations of the solutions, obtained in the scale limit $\epsilon \to 0$, are highly demanded. Homogenization results for drift-diffusion models in porous media usually deal with the Poisson-Nernst-Planck (PNP) system or with models consisting of the PNP system coupled with the Stokes system (SPNP). Formal upscaling of PNP system using formal asymptotic expansions is given, e.g., in \cite{Aur94, Moy06}. Rigorous homogenization results have been obtained for the SPNP system in the case of two charged species with opposite valences, e.g., in \cite{Schmu11, Ray12}. Furthermore, in \cite{All10} a stationary and linearized SPNP for multiple charged species was homogenized. In the recent paper \cite{Kov21}, a generalized PNP problem in a two-phase medium with transmission conditions at the microscopic interface has been homogenized. Here multiple charged species were considered, however under the constraint of \textit{total mass balance}. Corrector results related to the model from \cite{Ray12} were considered in \cite{Kho19}. Homogenization results for drift-diffusion system for multiple charged species without additional constraints (e.g., linearization, \textit{total mass balance}) are not available in the literature so far. This might be related to the fact that existence results in dimension three for such problems are only available in very weak function spaces, see, e.g., \cite{Both14}.  In contrast, problem \eqref{non_dim_PNP}-\eqref{def_h_p} has (for fixed values of the parameters $\epsilon>0$ and $\eta>0$) weak solutions with good regularity properties, especially for the potential, see Proposition \ref{thm_exist} below for details. Thus, this problem is more suitable (than the classical PNP problem) for the investigation of the behavior of the solutions in the limit $\epsilon \to 0$, and for the derivation of an effective (macroscopic) approximation in case of multiple charged species without imposing further constraints. 

From the point of view of multiscale analysis and homogenization the problem \eqref{non_dim_PNP}-\eqref{def_h_p} (for fixed values of $\eta>0$) is of independent interest due to the nonlinearity in the diffusion term and the strong nonlinear coupling via the drift-term. In the literature there are few contributions dealing with the homogenization of quasi-linear problems and nonlinear diffusion, we mention here, e.g., \cite{Luc96, Lis20, Gah21}. In \cite{Luc96} the rigorous upscaling of a two phase flow in a perforated domain was performed while in \cite{Lis20} fluid flow in an unsaturated porous
medium containing a fracture was upscaled. In \cite{Gah21} a reaction-diffusion problem with nonlinear diffusion was homogenized in a domain consisting of two bulk regions separated by a thin layer with periodic microstructure. Further, problems including monotone operators were treated in \cite{All92} for the stationary case and in \cite{ClarkShowalter1999} for nonstationary problems. 
As usually in the homogenization of nonlinear problems, the main challenge of our study is to derive \textit{a priori} estimates of the solutions uniformly with respect to the scale parameter $\epsilon$, which allow to pass to the limit $\epsilon \to 0$, especially in the nonlinear terms. In particular, an $L^\infty$-estimate in time and space of the concentration vector $(c_{i, \epsilon})_{1\leq i\leq P}$
is needed. To obtain this estimate, we make use of an energy functional associated to the system, which allows us to estimate the $L^\infty L^p$-norm of the concentrations. The energy functional is inspired from \cite{Both14}, where the existence of solutions of problem  \eqref{non_dim_PNP}-\eqref{def_h_p} with Robin-type boundary conditions for the potential was shown.  The estimate of the concentration vector then allows to prove an $L^\infty H^{1}$-estimate of the potential $\phi_\epsilon$ and eventually, to obtain the $L^\infty L^\infty$-estimate for the concentrations by using a classical result from \cite{Lad68}. A key point in our proof is to avoid the use of higher (than $H^{1}$) norms for the potential. 
Based on the \textit{a priori} estimates effective (homogenized) problems are derived. It turns out that different results are obtained for different values of the parameters $\alpha$ and $\beta$. Whereas for $\alpha = \beta$ a drift-diffusion problem with homogenized coefficients (similar to the microscopic problem) is obtained, for $\alpha < \beta$ the homogenized problem reduces to a system of diffusion equations (with homogenized nonlinear diffusion) for the species concentrations, one way coupled to the homogenized Poisson's equation.  To our knowledge, the result obtained in this paper is the first one providing the rigorous homogenization of a three dimensional (nonlinear) electro-diffusion system for multiple charged species (avoiding further restrictions such as above). 

This paper is organized as follows. In Section \ref{Microscopic_model}, the microscopic model is introduced and existence of solutions is proved. In Section \ref{Sect_unif_est}, estimates for the microscopic solutions uniformly with respect to $\epsilon$ are derived. These are the basis for the (two-scale) convergence results for the microscopic solutions proved in Section \ref{Deriv_macro_model}. Also in this section the homogenized drift-diffusion model is derived.  The paper is concluded with a discussion and outlook in Section \ref{Discussion_and_outlook} and the appendix consisting of auxiliary results.


\section{The microscopic model}
\label{Microscopic_model}
	
We consider a porous medium occupying a bounded and connected domain $\Omega\subset \mathbb{R}^n, n=2,3$,  with $\partial \Omega$ of class $C^3$. The medium has a periodic microstructure, generated with the help of the scaled standard periodicity cell $Y=(0,1)^n$ which consists of a solid part $ Y^s$ and a fluid or pore part $ Y^f$. We assume that $ Y^s$ is an open set such that $\overline{Y^s}\subset Y$, and that the boundary $\Gamma := \partial Y^s$ is of class $C^3$. Furthermore, let $Y^f = Y\setminus \overline{Y^s}$, see also Figure 1 (left). For $k\in \mathbb{Z}^n$, let $Y_k := Y +k$ and $\Gamma_k:= \Gamma + k$. Furthermore, for $j=f, s$, set $Y_k^j:= Y^j + k$.  

For a given (small) scale parameter $\epsilon>0$, let
$
I_\epsilon = \left\{k \in \mathbb{Z}^n \left \vert \epsilon Y_k \subset \Omega \right. \right\}.
$
We define the microscopic domain $\Omega_{\epsilon}$ representing the pore part of the porous medium by
$$
\Omega_\epsilon = \Omega \setminus \bigcup_{k \in I_\epsilon} \epsilon \overline { Y_k^s},
$$
see also Figure 1 (right). We remark that the boundary of $ \Omega_\epsilon$ consists of two disjoint parts 
$$
\partial \Omega_\epsilon = \Gamma_\epsilon \cup \partial \Omega,
$$
where 
$
\Gamma_\epsilon := \bigcup_{k \in I_\epsilon} \epsilon \Gamma_k 
$
denotes the boundary of the microscopic solid grains. We also note that the domain $\Omega_\epsilon$ is connected with boundary $\partial \Omega_\epsilon$ of class $C^3$. 

The aim of the paper is the rigorous homogenization of the nonlinear drift-diffusion model \eqref{non_dim_PNP}-\eqref{def_h_p}, i.e., the derivation of a macroscopic model in the scale limit $\epsilon \to 0$. 


\subsection{Assumptions on the data}
\begin{itemize}[noitemsep]
\item[(A1)] For the diffusion coefficients we assume $D_i>0,  1 \leq i \leq P$. Furthermore, let $T>0$ be a fixed time point.
\item[(A2)] The surface charge density is given by
\begin{flalign*}
\xi_\epsilon(x)=
\begin{cases}
\epsilon \xi_1 (x, \frac{x}{\epsilon}) &\text{ if $x\in \Gamma_\epsilon$,} \\
\xi_2(x) &\text{ if $x\in \partial\Omega$,}
\end{cases}
\end{flalign*}
where $\xi_1 \in C^2(\overline{\Omega} \times \overline{\Gamma})$, $\xi_1(x,y)$ periodically extended with respect to $y$ with period $Y$, and $\xi_2 \in C^2(\partial \Omega)$. Let us denote
\begin{equation*}
	\xi^* := \max  \left \{\displaystyle \max_{(x,y) \in (\overline{\Omega} \times \overline{\Gamma})}| \xi_1 (x,y)|, \displaystyle  \max_{x \in \partial \Omega} |\xi_2 (x)| \right \}.  
		\end{equation*}	
\item[(A3)]	 For the initial concentrations we assume  $c^0_i \in C^2(\overline{\Omega})$ with $c^0_i \geq 0$ for $i=1, \ldots, P$.
\item [(A4)] We assume the following compatibility condition:
\begin{equation}\label{assum_data}
\int_{\Omega_\epsilon} \sum_{i=1}^{P} z_i c_i^0 (x) \,dx + \int_{\partial \Omega_\epsilon} \xi_\epsilon(x)\,dS(x) =0.
\end{equation}
\end{itemize}

\begin{remark}
Let us mention that the assumptions on the data and on the domain have to guarantee existence of a solution, but also they have to allow the passage to the homogenization limit. Correspondingly, we have to assume relatively hight regularity of the domain $\Omega_\epsilon$ and on the given charge density $\xi_\epsilon$ as well as on the initial concentration vector $c^0$, in order to show existence of microscopic solutions. The compatibility condition \eqref{assum_data} is also required for solvability of the microscopic model. To be able to perform the homogenization process, the periodicity assumption on the microstructure of the porous medium is fundamental. Taking a charge distribution $\xi_1(x, \frac{x}{\epsilon} )$ which depends on both a microscopic and a macroscopic variable is rather physical. The dependence on the second (microscopic) variable allows the presence of periodically repeating patterns in the charge distribution on the pores' boundaries, while the  dependence on the first (macroscopic) variable allows slight variations in these patterns between neighboring pores. 
\end{remark}


\subsection{Variational formulation of the microscopic problem}
The variational formulation of the problem \eqref{non_dim_PNP}-\eqref{def_h_p} is given as follows. Find non-negative functions $c_{i,\epsilon} \in L^\infty ((0,T)\times \Omega_\epsilon) \cap L^2(0,T; H^1(\Omega_\epsilon))$ with $ \partial_{t} c_{i,\epsilon} \in L^2(0,T; H^1(\Omega_\epsilon)^\prime)$ and $\phi_\epsilon \in L^\infty (0,T; W^{2,6} (\Omega_\epsilon)) $  with $\int _{\Omega_\epsilon} \phi_\epsilon (t,x) \,dx =0$ for all $t \in [0,T]$, satisfying 
\begin{eqnarray}\label{Exist_1}
 <\partial_{t} c_{i,\epsilon} ,\psi>_{H^1(\Omega_\epsilon)^\prime, H^1(\Omega_\epsilon) }+   \int_{\Omega_\epsilon} ( D_{i} \nabla h_p (c_{i,\epsilon}) + \epsilon^\beta D_{i} z_i c_{i,\epsilon} \nabla \phi_\epsilon) \nabla \psi  \,dx= 0,
 \end{eqnarray}
for all $\psi \in H^1(\Omega_\epsilon)$ and almost every $t \in (0,T)$, together with the initial condition
\begin{equation}\label{Exist_2}
c_{i,\epsilon}(0)= c_i^0 \quad \mbox{ in } L^2(\Omega_\epsilon),
\end{equation}
and 
\begin{equation}\label{weak_Poisson}
\epsilon ^\alpha \int_{\Omega_\epsilon} \nabla \phi_\epsilon (t) \nabla \upsilon \,dx = \int_{\Omega_\epsilon} \sum_{i=1}^{P} z_i c_{i,\epsilon}(t) \upsilon \,dx + \int_{\partial \Omega_\epsilon } \xi_\epsilon \upsilon\,dS(x),
 \end{equation}
for all  $\upsilon \in H^1(\Omega_\epsilon)$, and all $t \in [0,T]$.
\begin{remark}\label{Remark_Poisson}
We emphasize that due to the regularity properties of the variational solution, the Poisson's problem for the electric potential $\phi_\epsilon$ holds even pointwise almost everywhere in $x$ and $t$. 
\end{remark}
\begin{remark}
In order to keep the the notation as clear as possible, we skip the parameters $\alpha, \beta, \eta$ and $p$ in the labeling of the solution $(c_{i, \epsilon}, \phi_\epsilon)$.
\end{remark}


\subsection{Existence for the microscopic model}
In this section we prove the existence of solutions for the microscopic model \eqref{non_dim_PNP}-\eqref{def_h_p}. The proof follows from similar arguments as in \cite[Lemma 5]{Both14}, where a similar model was considered, however, with Poisson's equation for the potential subject to Robin-type boundary condition (instead of purely Neumann condition used in our model). Therefore, we mainly highlight the new arguments needed to prove the existence result for our setting. Note that in the proof we use   energy estimates from Proposition \ref{Thm: V_est} in the next section.

\begin{proposition}\label{thm_exist}
Suppose the assumptions $(A1)-(A4)$ are satisfied. Then there exists a solution $(c_{i,\epsilon}, \phi_\epsilon),  i \in \{1,...,P\}$ of the variational problem \eqref{Exist_1}-\eqref{weak_Poisson}. If $n=2$, additionally, we have $\phi_\epsilon \in L^\infty(0,T;W^{2,q}(\Omega_\epsilon))$ for any $q \in [1, \infty)$.
\end{proposition}
\begin{proof} The existence is obtained using Schaefer's fixed point theorem.
Consider any $\phi_\epsilon \in L^{\infty} (0,T; W^{1,\infty}(\Omega_\epsilon))$. For such a $\phi_\epsilon$ there exists a unique non-negative $c_{i,\epsilon} \in L^\infty ((0, T) \times \Omega_\epsilon )\cap L^2(0,T; H^1(\Omega_\epsilon)) $ with $\partial_{t} c_{i,\epsilon} \in L^2(0,T;H^1(\Omega_\epsilon)^\prime)$ satisfying (\ref{Exist_1}) and (\ref{Exist_2}), see \cite{Both14}. Note that here the $C^3$ regularity of the domain $\Omega_\epsilon$ is used. Again for these $c_{i, \epsilon}, 1\leq i \leq P$, there exists a unique $\hat{\phi}_\epsilon \in L^\infty (0,T; H^1(\Omega_\epsilon))$ with $\int_{\Omega_\epsilon} \hat{\phi}_\epsilon(t,x) \,dx =0$, for all $t \in [0,T]$ such that $\hat{\phi}_{\epsilon}$ satisfies Poisson's equation (\ref{weak_Poisson}), see \cite[Theorem 4.22]{Cio99}. We note that the compatibility condition 
\begin{equation}\label{existence_2}
\int_{\Omega_\epsilon} \sum_{i=1}^{P} z_i c_{i,\epsilon} (t,	x) \,dx + \int_{\partial \Omega_\epsilon} \xi_\epsilon\,dS(x) =0,\ \ \text{for all $t \in [0,T]$}
\end{equation}
required in the existence proof for $\hat{\phi}_\epsilon$ is satisfied due to the fact that testing equation (\ref{Exist_1}) with $\psi =1$, we obtain
\begin{equation*}
\frac{d}{dt} \int_{\Omega_\epsilon} c_{i,\epsilon}(t,x) \,dx =0 ,\ \ \text{for a.e. $t \in (0,T)$},
\end{equation*}
which together with assumption (A4) yields (\ref{existence_2}). 

Let us now consider the mapping 
\begin{equation}\label{tau}
\tau : L^\infty(0,T; W^{1,\infty} (\Omega_\epsilon)) \rightarrow L^\infty(0,T; W^{1,\infty} (\Omega_\epsilon)), \quad  \tau(\phi_\epsilon) = \hat{\phi}_\epsilon.
\end{equation}
Next, we show that $\hat{\phi}_\epsilon \in L^\infty(0,T;W^{1,\infty}(\Omega_\epsilon))$ so that $\tau$ becomes well-defined. Here again, we need to argue differently from \cite{Both14}. Since $c_{i,\epsilon} \in L^\infty(0,T; L^2(\Omega_\epsilon))$, elliptic regularity results imply that $\hat{\phi}_\epsilon \in L^\infty(0,T; H^2(\Omega_\epsilon))$, see \cite[Theorem 4, p. 217]{Mik78}. Consequently, using $c_{i,\epsilon} \in L^\infty((0,T) \times \Omega_\epsilon)$ and \cite[Lemma 2.4.1.4]{Gri85}, we obtain for almost every $t \in (0,T)$ that $\hat{\phi}_\epsilon(t) \in W^{2,6}(\Omega_\epsilon)$, if $n=3$, and $\hat{\phi}_\epsilon(t) \in W^{2,q}(\Omega_\epsilon)$ for any $q \in [1,\infty)$, if $n=2$. The $L^\infty$-regularity with respect to time, namely $\hat{\phi}_\epsilon \in L^\infty(0,T; W^{2,6}(\Omega_\epsilon))$, if $n=3$, and $\hat{\phi}_\epsilon \in L^\infty(0,T;W^{2,q}(\Omega_\epsilon))$ for any $q \in [1,\infty)$, if $n=2$, is then obtained from \cite[Theorem 2.4.1.3]{Gri85}. Thus, by Sobolev embedding theorem, we have $\hat{\phi} _\epsilon \in L^\infty(0,T;W^{1,\infty} (\Omega_\epsilon))$.

Note that a solution of the microscopic problem \eqref{Exist_1}-\eqref{weak_Poisson} is a fixed point of $\tau$. Thus, in the following, we prove the existence of such a fixed point by using Schaefer's fixed point theorem (see, e.g., \cite[Theorem 4, p. 504]{Eva10}). Since $W^{2,6}(\Omega_\epsilon)$ is compactly embedded in $W^{1,\infty}(\Omega_\epsilon)$, the continuity and compactness of $\tau$ follow by similar arguments as in \cite[Lemma 5]{Both14}. It remains to prove that the set 
\begin{equation}\label{existence_a_0}
 \{\phi_\epsilon \in L^\infty(0,T;W^{1,\infty}(\Omega_\epsilon)): \phi_\epsilon =\lambda \tau (\phi_\epsilon) \ \ \text{for some}\  \lambda \in [0,1]\}
\end{equation}
is bounded. Here we again need alternative arguments to \cite{Both14}. Suppose, $\tau(\phi_\epsilon) =\frac{\phi_\epsilon}{\lambda}$ for some $\lambda \in (0,1]$.  Then $\phi_\epsilon$ is a solution of the potential equation $\eqref{lambda_Poiss}$ in Proposition \ref{Thm: V_est}. Firstly, let us show that $\tau(\phi_\epsilon)$ is bounded in $L^\infty(0,T;W^{2,4}(\Omega_\epsilon))$ independently of $\lambda$. By Proposition \ref{Thm: V_est} and Corollary \ref{Remark_log_bound}, we can bound the $L^\infty(0,T;L^p(\Omega_\epsilon))$ norm of $c_{i,\epsilon}$ independently of $\lambda$. 
Since by assumption, see \eqref{def_h_p}, we have $p \geq 4$, it follows that $c_{i,\epsilon}$ is bounded in $L^\infty(0,T;L^4(\Omega_\epsilon))$ independent of $\lambda$.
Hence, there exists a constant $C_1 >0$ independent of $\lambda$ such that
\begin{equation}\label{existence_3}
 \left \lVert \sum_{i=1}^{P} z_i c_{i,\epsilon} \right \rVert _{L^\infty(0,T;L^4(\Omega_\epsilon))} < C_1.
\end{equation}
Testing with $\tau(\phi_\epsilon)(t)$ in the weak formulation satisfied by $\tau(\phi_\epsilon)(t)$, we get
\begin{eqnarray*}
 \epsilon^\alpha \lVert \nabla \tau (\phi_\epsilon)(t) \rVert^2_{L^2(\Omega_\epsilon)} &&\leq \left \lVert \sum_{i=1}^{P} z_ic_{i,\epsilon}(t) \right \rVert_{L^2(\Omega_\epsilon)} \lVert \tau (\phi_\epsilon)(t) \rVert _{L^2(\Omega_\epsilon)}\\
 &&+ \lVert \xi_\epsilon \rVert _{L^2(\partial \Omega_\epsilon)} \lVert \tau (\phi_\epsilon) (t) \rVert _{L^2(\partial \Omega_\epsilon)}\\
 && \leq C_2 \left( \left \lVert \sum_{i=1}^{P} z_ic_{i,\epsilon}(t) \right \rVert_{L^2(\Omega_\epsilon)}+\lVert \xi_\epsilon \rVert_{L^2(\partial \Omega_\epsilon)}\right) \lVert \nabla \tau (\phi_\epsilon)(t) \rVert_{L^2(\Omega_\epsilon)}.
\end{eqnarray*}
In the last inequality we have used $\int_{\Omega_\epsilon} \tau(\phi_\epsilon)(t) =0$. Consequently, we have
\begin{equation}\label{existence_a}
 \lVert \tau(\phi_\epsilon)(t) \rVert _{H^1(\Omega_\epsilon)} \leq C_3  \left(   \left \lVert \sum_{i=1}^{P} z_ic_{i,\epsilon}(t) \right \rVert_{L^2(\Omega_\epsilon)}+\lVert \xi_\epsilon \rVert_{L^2(\partial \Omega_\epsilon)} \right ).
\end{equation}
Using (\ref{existence_3}) in (\ref{existence_a}) we obtain that, $\lVert \tau (\phi_\epsilon)\rVert_{
L^\infty(0,T;H^1(\Omega_\epsilon))}$ is bounded independent of $\lambda$. Due to the embedding $H^1(\Omega_\epsilon) \subset L^4(\Omega_\epsilon)$, we have for some $C_4>0$ independent of $\lambda$:
\begin{equation} \label{existence_b}
 \lVert \tau (\phi_\epsilon)\rVert_{
L^\infty(0,T;L^4(\Omega_\epsilon))} <C_4.
\end{equation}
Again, with the help of \cite[Theorem 2.4.1.3]{Gri85}, for all $t \in [0,T]$, we obtain
\begin{equation*}
 \lVert \tau(\phi_\epsilon)(t) \rVert_{W^{2,4}(\Omega_\epsilon)} \leq C_5\left(\left\lVert  \epsilon^{-\alpha}  \sum_{i=1}^{P}  z_i c_{i,\epsilon}(t) + \mu \tau(\phi_\epsilon)(t) \right \rVert_{L^4(\Omega_\epsilon)} + \epsilon^{-\alpha} \lVert \xi_\epsilon \rVert_{W^{1-\frac{1}{4},4}(\partial \Omega_\epsilon)} \right ),
\end{equation*}
for some $C_5, \mu >0$ independent of $\lambda$. This implies
\begin{equation}\label{existence_c}
 \lVert \tau(\phi_\epsilon)(t) \rVert_{W^{2,4}(\Omega_\epsilon)} \leq C_6 \left( \left \lVert \sum_{i=1}^{P} z_ic_{i,\epsilon}(t) \right \rVert_{L^4(\Omega_\epsilon)} + \lVert \tau(\phi_\epsilon)(t) \rVert _{L^4(\Omega_\epsilon)} +\lVert \xi_\epsilon \rVert_{W^{1-\frac{1}{4},4}(\partial \Omega_\epsilon)}\right ),
\end{equation}
for some $C_6>0$ independent of $\lambda$. Finally using (\ref{existence_3}), (\ref{existence_b}) in (\ref{existence_c}) we obtain the desired result that $\tau(\phi_\epsilon)$ is bounded in $L^\infty(0,T;W^{2,4}(\Omega_\epsilon))$ independent of $\lambda$.

Consequently, the embedding $W^{2,4}(\Omega_\epsilon) \subset W^{1,\infty} (\Omega_\epsilon)$, implies that $\tau (\phi_\epsilon)$ is bounded in $L^\infty(0,T;W^{1,\infty}(\Omega_\epsilon))$ independent of $\lambda$. Since $\phi_\epsilon =\lambda \tau(\phi_\epsilon)$, for some $\lambda \in (0,1]$, we conclude that $\phi_\epsilon$ is bounded in $L^\infty(0,T;W^{1,\infty}(\Omega_\epsilon))$ independent of $\lambda$ and the proof is complete.
\end{proof}


\section{Uniform estimates for the microscopic solutions}
\label{Sect_unif_est}
In this section, we prove the uniform estimates for the microscopic solutions. To obtain compactness results which allow to pass to the limit in the nonlinear terms, we prove an $L^\infty$-estimate for the concentrations, uniformly with respect to $\epsilon$. The first step towards this result is to show that the concentrations are uniformly bounded in $L^\infty((0,T),L^p(\Omega_\epsilon))$. (Note that in the whole paper $p\in [4, \infty)$ is a fixed index entering the definition \eqref{def_h_p} of the nonlinear diffusion function $h_p$.) For this we use the following energy functional associated to our system, see also \cite{Both14,Both14a}:
\begin{equation}\label{est_V_0}
V_\epsilon(t)=\frac{1}{2\lambda}  \epsilon^{\alpha+\beta} \int_{\Omega_\epsilon} |\nabla \phi_\epsilon |^2 \,dx+\sum_{i=1}^{P} \int_{\Omega_\epsilon} \Psi(c_{i,\epsilon}) \,dx  ,\ \ \text{for a.e. $t \in (0,T)$},
\end{equation}
where $\lambda \in (0,1]$, and 
\begin{equation}\label{def_Psi}
\Psi(r) = r \log r -r + 1+ \frac{\eta}{p-1}r^p,\ \ \text{for $r \geq 0$.}
\end{equation}
%
\begin{proposition}\label{Thm: V_est}
Let $\lambda \in (0,1]$. Consider non-negative functions $c_{i,\epsilon} \in L^\infty ((0,T)\times \Omega_\epsilon) \cap L^2(0,T; H^1(\Omega_\epsilon))$ with $ \partial_{t} c_{i,\epsilon} \in L^2(0,T; H^1(\Omega_\epsilon)^\prime)$, and $\phi_\epsilon \in L^\infty(0,T;W^{2,6}(\Omega_\epsilon))$ with $ \int_{\Omega_\epsilon} \phi_\epsilon (t,x)\,dx =0$ for all $t \in [0,T] $ that satisfy (\ref{Exist_1})-(\ref{Exist_2}) and the following equations:
\begin{align}\label{lambda_Poiss}
	&& -\epsilon^{\alpha} \Delta  \phi _\epsilon(t,x) &=  \lambda \sum_{i=1}^{P}  z_i c_{i,\epsilon} (t,x) &&\text{$\forall t \in [0,T]$ and a.e. $x \in   \Omega_\epsilon,$}
	\end{align}
\begin{align}\label{lambda_Poiss_BC}
	&& \epsilon^{\alpha} \nabla  \phi _\epsilon(t,x) \cdot \nu_\epsilon &= \lambda \xi_\epsilon(x) &&\text{$\forall t \in [0,T]$ and a.e. $x \in \partial \Omega_\epsilon.$}
	\end{align}
Then 
\begin{equation}\label{V_est_9}
 \frac{d}{dt} V_\epsilon(t) \leq 0, \ \ \mbox{for a.e. } t \in (0,T).
\end{equation}
In particular, we have
\begin{equation}\label{est_V_0a}
V_\epsilon(t) \leq C \left(1+\epsilon^ {-\alpha+\beta}  \right), \ \ \text{for all $t \in [0,T]$,}
\end{equation}
with some constant $C>0$ independent of $\epsilon$ and $\lambda$. 
\end{proposition}
 \begin{proof} 
Using $<\partial_{t} c_{i,\epsilon}, 1>_{H^1(\Omega_\epsilon)^\prime,H^1(\Omega_\epsilon)} =0$, we have
\begin{equation}\label{V_est}
\frac{d}{dt}V_\epsilon(t)= \frac{d}{dt}\left(\frac{1}{2\lambda}\epsilon^{\alpha+\beta}\int_{\Omega_\epsilon}|\nabla \phi_\epsilon |^2\right) +\frac{d}{dt}\left(\sum_{i=1}^{P}  \int_{\Omega_\epsilon} c_{i,\epsilon} \log c_{i,\epsilon}\right)+\frac{d}{dt}\left(\sum_{i=1}^{P}  \int_{\Omega_\epsilon}\frac{\eta}{p-1}c_{i,\epsilon}^p \right),
\end{equation}
for almost every $t \in (0,T)$, provided that each of the three time derivatives given in the right hand side of (\ref{V_est}) exists. Next we show that this is the case. 
Suppose $\delta_m$ is a sequence of positive numbers tending to zero as $m$ tends to infinity. Then from Lemma \ref{lemma_weak_time_derivative},  
we have for all $m \in \N$ and $t_1 \in [0,T]$,
\begin{eqnarray} \label{V_est_1}
&&\int_{\Omega_\epsilon} [c_{i,\epsilon}(t_1) + \delta_m] \log [c_{i,\epsilon}(t_1) + \delta_m] \,dx - \int_{\Omega_\epsilon}[c_{i,\epsilon}(0) + \delta_m] \log [c_{i,\epsilon}(0) + \delta_m] \,dx \nonumber\\
&&= \int_{0}^{t_1} \left < \partial _{t} c_{i,\epsilon} , \log (c_{i,\epsilon} + \delta_m) \right>_{{H^1(\Omega_\epsilon)^\prime,H^1(\Omega_\epsilon)}} \,dt.
\end{eqnarray}
Testing (\ref{Exist_1}) with $\log (c_{i,\epsilon} + \delta_m)$ we see that the right hand side of (\ref{V_est_1}) is equal to 
\begin{eqnarray*}
&& \int_{0}^{t_1} \int_{\Omega_\epsilon} J_{i, \epsilon} \cdot \frac{\nabla c_{i,\epsilon}}{c_{i,\epsilon}+\delta_m}dx dt \\
&&=  \int_{0}^{t_1} \int_{\Omega_\epsilon}\frac{1-h_p^ \prime (c_{i,\epsilon})}{c_{i,\epsilon}+\delta_m} \ \nabla c_{i,\epsilon} \cdot J_{i, \epsilon} + \frac{\nabla h_p(c_{i,\epsilon})}{c_{i,\epsilon}+\delta_m} \cdot J_{i, \epsilon}dx dt \\ 
&&=  \int_{0}^{t_1} \int_{\Omega_\epsilon}\frac{1-h_p^ \prime (c_{i,\epsilon})}{c_{i,\epsilon}+\delta_m} \ \nabla c_{i,\epsilon} \cdot J_{i, \epsilon} + \frac{ D_{i} \nabla h_p (c_{i,\epsilon})+\epsilon^\beta D_{i} z_i c_{i,\epsilon} \nabla \phi_\epsilon}{D_{i}(c_{i,\epsilon}+ \delta_m)} \cdot J_{i, \epsilon}dx dt \\
&&-\int_{0}^{t_1} \int_{\Omega_\epsilon} \frac{\epsilon^\beta z_i c_{i,\epsilon}\nabla \phi_\epsilon}{c_{i,\epsilon}+\delta_m} \cdot J_{i, \epsilon}dx dt \\
&&=  \int_{0}^{t_1} \int_{\Omega_\epsilon} - \frac{\eta p c_{i,\epsilon}^{p-1}}{c_{i,\epsilon}+\delta_m} \nabla c_{i,\epsilon} \cdot J_{i, \epsilon} - \frac{|J_{i, \epsilon}|^2}{D_{i}(c_{i,\epsilon}+\delta_m)} - \frac{\epsilon^\beta z_i c_{i,\epsilon}\nabla \phi_\epsilon}{c_{i,\epsilon}+\delta_m} \cdot J_{i, \epsilon} dx dt .
\end{eqnarray*}
Hence from (\ref{V_est_1}), we get
\begin{eqnarray}\label{V_est_2}
\int_{\Omega_\epsilon} [c_{i,\epsilon}(t_1) + \delta_m] \log [c_{i,\epsilon}(t_1) + \delta_m] - [c_{i,\epsilon}(0) + \delta_m] \log [c_{i,\epsilon}(0) + \delta_m] dx \nonumber \\
= \int_{0}^{t_1} \int_{\Omega_\epsilon} - \frac{\eta p c_{i,\epsilon}^{p-1}}{c_{i,\epsilon}+\delta_m} \nabla c_{i,\epsilon} \cdot J_{i, \epsilon} - \frac{|J_{i, \epsilon}|^2}{D_{i}(c_{i,\epsilon}+\delta_m)} - \frac{\epsilon^\beta z_i c_{i,\epsilon}\nabla \phi_\epsilon}{c_{i,\epsilon}+\delta_m} \cdot J_{i, \epsilon}dx dt .
\end{eqnarray}
Again, using the monotone convergence theorem, we have
\begin{equation}\label{V_est_3}
\lim_{m \to \infty}\int_{0}^{t_1} \int_{\Omega_\epsilon} \frac{|J_{i, \epsilon}|^2}{D_{i}(c_{i,\epsilon}+\delta_m)}dx dt = \int_{0}^{t_1} \int_{\Omega_\epsilon} \frac{|J_{i, \epsilon}|^2}{D_{i}c_{i,\epsilon}}dx dt.
\end{equation}
However, this limit can be infinite. Next, we shall show that this is not the case. 

The inequalities
\begin{equation*}
x \log x \geq -1 ,\ \ \forall x>0 \ \ \text{and};
\end{equation*}
\begin{flalign*}
x \log x \leq
\begin{cases}
0 &\text{ if $x\in (0,1)$,} \\
x^2 &\text{ if $x \geq 1$;}
\end{cases}
\end{flalign*}
together give the following:
\begin{equation*}
|(c_{i,\epsilon}+\delta_m) \log(c_{i,\epsilon}+\delta_m)| \leq 1+ (c_{i,\epsilon}+\delta_m)^2 \leq 1+ (c_{i,\epsilon}+1)^2 \ \ \text{(for large $m$),}
\end{equation*}
for all $t \in [0,T]$ and almost every $x \in \Omega_\epsilon$. For each $t$, we have $1+ (c_{i,\epsilon}+1)^2 \in L^1(\Omega_\epsilon)$. This fact and the continuity of the function $ x  \mapsto x \log x$ for all $x\geq 0$ allow us to use the dominated convergence theorem to obtain the following:
\begin{equation}\label{V_est_4}
\lim_{m \to \infty} \int_{\Omega_\epsilon} (c_{i,\epsilon}+\delta_m) \log (c_{i,\epsilon}+\delta_m)dx = \int_{\Omega_\epsilon}c_{i,\epsilon} \log c_{i,\epsilon} dx,  \ \ \forall t \in [0,T].
\end{equation}
Using (\ref{V_est_3}), (\ref{V_est_4}), the fact that
\begin{eqnarray*}
0 \leq \frac{c_{i,\epsilon}}{c_{i,\epsilon}+ \delta_m} < 1; \ \text{and as $m \to \infty$,} \ 	
	 \frac{c_{i,\epsilon}}{c_{i,\epsilon}+ \delta_m} \to \chi_{\{c_{i,\epsilon} \neq 0\} }\ \text{a.e. in $(0,T) \times \Omega_\epsilon$;}
\end{eqnarray*}advective
we pass to the limit in (\ref{V_est_2}):
\begin{eqnarray}\label{V_est_a}
 \sum_{i=1}^{P}\int_{\Omega_\epsilon}  c_{i,\epsilon}(t_1)  \log c_{i,\epsilon}(t_1)  -  c_{i,\epsilon}(0) \log c_{i,\epsilon}(0) dx \nonumber \\
= \sum_{i=1}^{P}\int_{0}^{t_1} \int_{\Omega_\epsilon}  -\, \eta p c_{i,\epsilon}^{p-2} \nabla c_{i,\epsilon} \cdot J_{i, \epsilon} - \frac{|J_{i, \epsilon}|^2}{D_{i}c_{i,\epsilon}} - \epsilon^\beta z_i \chi_{\{c_{i,\epsilon} \neq 0\}} \nabla \phi_\epsilon \cdot J_{i, \epsilon} dx dt ,\nonumber\\
\  \text{$\forall t_1 \in [0,T]$.}
\end{eqnarray}
We note from (\ref{V_est_a}) that $\sum_{i=1}^P\frac{|J_{i, \epsilon}|^2}{D_{i}c_{i,\epsilon}} \in L^1((0,t_1) \times \Omega_\epsilon)$ and hence the limit in (\ref{V_est_3}) is finite. Now from  (\ref{V_est_a}), we see that $ \sum_{i=1}^P \int_{\Omega_\epsilon} c_{i,\epsilon} \log c_{i,\epsilon}$ is absolutely continuous on $[0,T]$ and
\begin{eqnarray}\label{V_est_5}
&&\frac{d}{dt} \sum_{i=1}^{P} \int_{\Omega_\epsilon}  c_{i,\epsilon}  \log c_{i,\epsilon} dx \nonumber \\
&&= \sum_{i=1}^{P} \int_{\Omega_\epsilon}  - \,\eta p c_{i,\epsilon}^{p-2} \nabla c_{i,\epsilon} \cdot J_{i, \epsilon} - \frac{|J_{i, \epsilon}|^2}{D_{i}c_{i,\epsilon}} - \epsilon^\beta z_i \chi_{\{c_{i,\epsilon} \neq 0\}}\nabla \phi_\epsilon \cdot J_{i, \epsilon}dx, \nonumber \\
\end{eqnarray}
for almost every $t \in (0,T)$.

Next, consider the first term of the right hand side of (\ref{V_est_5}). Using (\ref{Exist_1}) we have

\begin{eqnarray}\label{V_est_6}
&&\sum_{i=1}^{P} \int_{\Omega_\epsilon}  \eta p c_{i,\epsilon}^{p-2} \nabla c_{i,\epsilon} \cdot J_{i, \epsilon}dx =\sum_{i=1}^{P}  \int_{\Omega_\epsilon} \nabla \left ( \frac{\eta p}{p-1} c_{i,\epsilon}^{p-1}\right) \cdot J_{i, \epsilon}dx \nonumber \\
&& =\sum_{i=1}^{P} \left < \partial _{t} c_{i,\epsilon} , \frac{\eta p}{p-1} c_{i,\epsilon}^{p-1} \right >_ {H^1(\Omega_\epsilon)^\prime,H^1(\Omega_\epsilon)} = \sum_{i=1}^{P}  \frac{d}{dt} \int_{\Omega_\epsilon }\frac{\eta }{p-1} c_{i,\epsilon}^{p}dx,
\end{eqnarray}
where the last equality can be proven using approximation arguments similar to Lemma \ref{lemma_weak_time_derivative}. 

Since $\nabla c_{i,\epsilon}= 0$ almost everywhere on the set $A  := \{(t,x) \in (0,T) \times \Omega_\epsilon: c_{i,\epsilon}  (t,x)= 0\}$, we see that $J_{i, \epsilon} $ vanishes on $A$. Using this the last term of the right hand side of (\ref{V_est_5}) becomes
\begin{eqnarray}-\label{V_est_7}
\sum_{i=1}^{P} \int_{\Omega_\epsilon}
\epsilon^\beta z_i \chi_{\{c_{i,\epsilon} \neq 0\}}\nabla \phi_\epsilon \cdot J_{i, \epsilon}dx &&=- \sum_{i=1}^{P} \int_{\Omega_\epsilon}
\epsilon^\beta z_i \nabla \phi_\epsilon \cdot J_{i, \epsilon} dx\nonumber \\
&&= -\epsilon^\beta \left <\sum_{i=1}^P z_i\partial_{t}c_{i,\epsilon}, \phi_\epsilon \right>_{H^1(\Omega_\epsilon)^\prime,H^1(\Omega_\epsilon)}.
 \end{eqnarray}
 Differentiating (\ref{lambda_Poiss}), (\ref{lambda_Poiss_BC})  with respect to $t$, we see that $\partial_{t} \phi_\epsilon \in H^1(\Omega_\epsilon)$ and the last term in (\ref{V_est_7}) is equal to
 \begin{eqnarray}\label{V_est_8}
 && -\epsilon^\beta \frac{1}{\lambda} \int_{\Omega_\epsilon} \epsilon^\alpha \nabla \partial_{t}   \phi_\epsilon    \nabla  \phi_\epsilon  dx  =- \frac{1}{2\lambda}\epsilon^{\alpha + \beta}\frac{d}{dt} \int_{\Omega_\epsilon} |\nabla \phi_\epsilon|^2 dx,
  \end{eqnarray}
  where we used $\partial_{t} \xi_\epsilon =0$, since $\xi_\epsilon$ is independent of $t$.
  From (\ref{V_est_5}), (\ref{V_est_6}), (\ref{V_est_8}) we conclude that $V_\epsilon$ is differentiable with respect to time and (\ref{V_est}) holds. Moreover, using (\ref{V_est_5}) (\ref{V_est_6}), (\ref{V_est_7}), (\ref{V_est_8})  in (\ref{V_est}), we have 
  \begin{equation*}
   \frac{d}{dt} V_\epsilon(t)+\sum_{i=1}^P\int_{\Omega_\epsilon}\frac{|J_{i, \epsilon}|^2}{D_{i,\epsilon}c_{i,\epsilon}} dx =0, \ \ \text{for a.e. $t \in (0,T)$.}
  \end{equation*}
Hence, 
\begin{equation*}
 \frac{d}{dt} V_\epsilon(t) \leq 0, \ \text{for a.e. $t \in (0,T)$,} 
\end{equation*}
what implies
\begin{align*}
V_\epsilon(t) \leq V_\epsilon(0) \qquad \mbox{ for all } t \in [0,T].
\end{align*}

It remains to obtain uniform estimate of $V_\epsilon(0)$.
\begin{eqnarray}\label{V_est_10}
 V_\epsilon(0) &&=\underbrace{\frac{1}{2\lambda}  \epsilon^{\alpha+\beta} \int_{\Omega_\epsilon} |\nabla \phi_\epsilon (0)|^2 \,dx}_\text{$:=I_1$} \nonumber  \\
 &&+ \underbrace{\sum_{i=1}^{P} \int_{\Omega_\epsilon}  \left(c_{i}^0  \log c_{i}^0 -c_{i}^0 +1 + \frac{\eta}{p-1}(c_{i}^0) ^p\right )\,dx}_\text{$:= I_2$}.
 \end{eqnarray}
 Now,
\begin{equation}\label{V_est_10_a}
 | I_2 | \leq  |\Omega | \sum_{i=1}^{P} \max _{c_i^0 \in \overline{\Omega}} \left | c_{i}^0  \log c_{i}^0 -c_{i}^0 +1 + \frac{\eta}{p-1}(c_{i}^0) ^p \right | \leq C  , 
\end{equation}
for some $C >0$ independent of $\epsilon$.

Now it remains to estimate $I_1$. Testing with $\phi_\epsilon(0)$ the weak formulation of (\ref{lambda_Poiss})-(\ref{lambda_Poiss_BC}) at $t=0$ , we get 
\begin{eqnarray*}
&&\epsilon^\alpha \int_{\Omega_\epsilon} | \nabla \phi_\epsilon (0) | ^2  \,dx\\
&&\leq C \lambda \left [  \int_{\Omega_\epsilon} |\phi_\epsilon (0) | \,dx+  \epsilon \int_{\Gamma_\epsilon}  |\phi_\epsilon (0)| \,dS(x) + \int_{\partial \Omega}  |\phi_\epsilon (0)| \,dS(x) \right ].
\end{eqnarray*}
Since $|\Omega_\epsilon|= O(1)$ and $| \Gamma_\epsilon| = O \left (\frac{1}{\epsilon}\right)$, we have 
\begin{eqnarray}\label{phi_0_est_1}
 \epsilon^\alpha \lVert \nabla \phi_\epsilon(0) \rVert ^2_{L^2(\Omega_\epsilon)}  \leq C \lambda \left [\lVert \phi_\epsilon(0) \rVert_{L^2(\Omega_\epsilon)} + \sqrt{\epsilon} \lVert \phi_\epsilon(0) \rVert_{L^2(\Gamma_\epsilon)} + \lVert \phi_\epsilon(0) \rVert_{L^2(\partial \Omega)}\right].
\end{eqnarray}
From the properties of the extension operator in Lemma \ref{lemma_extension} given in the appendix and the usual trace-inequality we obtain
\begin{equation}\label{phi_0_est_3}
 \lVert \phi_\epsilon(0) \rVert_{L^2(\partial \Omega)} = \lVert \widetilde{\phi_\epsilon}(0) \rVert_{L^2(\partial \Omega)} \leq C (\Omega) \lVert \widetilde{\phi_\epsilon}(0) \rVert_{H^1( \Omega)} \leq C_1 \lVert \phi_\epsilon(0) \rVert_{H^1( \Omega_\epsilon)},
\end{equation}
together with the scaled trace-inequality from Lemma \ref{TraceInequalityLemma} from the appendix we obtain from  $\eqref{phi_0_est_1}$  
\begin{eqnarray*}
 \epsilon^\alpha \lVert \nabla \phi_\epsilon(0) \rVert ^2_{L^2(\Omega_\epsilon)}  \leq C \lambda \lVert \phi_\epsilon(0) \rVert_{H^1( \Omega_\epsilon)}.
\end{eqnarray*}
Then (\ref{Lemma_mean_value}) leads to 
\begin{equation*}
 \epsilon^\alpha \lVert \nabla \phi_\epsilon(0) \rVert _{L^2(\Omega_\epsilon)}  \leq C \lambda.
\end{equation*}
Since $\lambda \in (0,1]$, we obtain
\begin{equation}\label{V_est_10_b}
 I_1 \leq C \epsilon ^ {-\alpha + \beta}.
\end{equation}
(\ref{V_est_10_a}) and (\ref{V_est_10_b}) complete the proof.
\end{proof}
\begin{corollary}\label{Remark_log_bound}
Due to the fact that  $r \log r -r +1 \geq 0$, for $r \geq 0$, and taking into account that $\alpha \leq \beta$, Proposition \ref{Thm: V_est} immediately implies
\begin{equation}\label{EstLinftyLp}
 \lVert c_{i,\epsilon}  \rVert_{L^\infty(0,T;L^p(\Omega_\epsilon))} \leq C, 
\end{equation}
with $C>0$ independent of $\lambda$ and $\epsilon$.
\end{corollary}
Based on estimate \eqref{EstLinftyLp}, in the following theorem we derive energy estimates for the microscopic solutions $(c_{i,\epsilon},\phi_\epsilon)$.
\begin{proposition}\label{Thm_energ_est_1}
 There exists a constant $C>0$ independent of $\epsilon$ such that the microscopic solutions $(c_{i,\epsilon},\phi_\epsilon)$ of (\ref{Exist_1})-(\ref{weak_Poisson}) satisfy the following estimates:
 \begin{equation}\label{Thm_energy_est_1}
\epsilon ^\alpha \lVert \phi_\epsilon \rVert_{L^\infty(0,T;L^2(\Omega_\epsilon))} + \epsilon ^\alpha \lVert \nabla  \phi_\epsilon \rVert_{L^\infty(0,T;L^2(\Omega_\epsilon))} \leq C,
 \end{equation}
 \begin{equation}\label{Thm_energy_est_2}
\lVert c_{i,\epsilon} \rVert _{L^\infty(0,T;L^p(\Omega_\epsilon))} + \lVert \nabla c_{i,\epsilon} \rVert_{L^2(0,T;L^2(\Omega_\epsilon))}\leq C.
\end{equation}
\end{proposition}
\begin{proof}
 Testing the weak formulation of Poisson's equation (\ref{weak_Poisson}) with $\phi_{\epsilon}$, we get almost everywhere in $(0,T)$
 \begin{eqnarray} \label{energ_est_1}
  &&\epsilon^\alpha \lVert \nabla \phi_\epsilon  \rVert ^2_{L^2(\Omega_\epsilon)}\nonumber \\
  && \leq \left \lVert \sum_{i=1}^{P} z_i c_{i,\epsilon} \right \rVert_{L^2(\Omega_\epsilon)} \lVert \phi_\epsilon  \rVert _{L^2(\Omega_\epsilon)}  +C \epsilon \int_{\Gamma_\epsilon} | \phi_\epsilon  | \,dS(x)+ C\int_{\partial \Omega}| \phi_\epsilon  | \,dS(x) \nonumber\\
  && \leq C \lVert \phi_\epsilon \rVert _{L^2(\Omega_\epsilon)} + C \sqrt{\epsilon} \lVert \phi_\epsilon   \rVert _{L^2(\Gamma_\epsilon)} +C \lVert \phi_\epsilon   \rVert _{L^2(\partial  \Omega)};
 \end{eqnarray}
here in the last inequality we used that $\sum_{i=1}^{P} z_i c_{i,\epsilon}$ is  bounded in $L^\infty(0,T;L^2(\Omega_\epsilon))$ uniformly with respect to $\epsilon$ (see Proposition \ref{Thm: V_est} \& Corollary \ref{Remark_log_bound}) and $| \Gamma_\epsilon |=O\left (\frac{1}{\epsilon}\right)$. Comparing (\ref{energ_est_1}) to (\ref{phi_0_est_1}) with $\lambda =1$, we conclude that the estimate $\lVert  \nabla \phi_\epsilon (t)\rVert _{L^2(\Omega_\epsilon)}$ follows from similar arguments that we used to estimate for $\lVert  \nabla \phi_\epsilon (0)\rVert _{L^2(\Omega_\epsilon)}$ in Proposition \ref{Thm: V_est}.
Also, we have
\begin{equation}\label{energ_est_2}
\epsilon^\alpha \lVert \nabla \phi_\epsilon  \rVert _{L^\infty(0,T;L^2( \Omega_\epsilon))} \leq C,
\end{equation}
for some $C>0$ independent of $\epsilon$. Consequently, using (\ref{Lemma_mean_value}) and (\ref{energ_est_2}), we get
\begin{equation*}
\epsilon^\alpha \lVert  \phi_\epsilon \rVert _{L^\infty(0,T; L^2(\Omega_\epsilon))} \leq C.
\end{equation*}
Next we prove (\ref{Thm_energy_est_2}). Due to Corollary \ref{Remark_log_bound} we only have to estimate the norm of the gradient. We test the equation (\ref{Exist_1}) with $c_{i,\epsilon}$ and get
\begin{eqnarray*}
 &&<\partial_{t} c_{i,\epsilon} , c_{i,\epsilon} >_{H^1(\Omega_\epsilon)^\prime, H^1(\Omega_\epsilon)} + D_i \int_{\Omega_\epsilon} |\nabla c_{i,\epsilon}|^2   \,dx\\
 &&+ D_i \eta p \int_{\Omega_\epsilon}  c_{i,\epsilon}^{p-1}  | \nabla c_{i,\epsilon}|^2\,dx + \epsilon ^ \beta D_{i}z_{i}\int_{\Omega_\epsilon} c_{i,\epsilon} \nabla \phi_{\epsilon} \nabla c_{i,\epsilon} \,dx = 0.
\end{eqnarray*}
Since the third term of the above equation is non-negative, we have
\begin{equation}\label{energ_est_3}
 \frac{1}{2} \frac{d}{dt} \lVert c_{i,\epsilon} \rVert ^2_{L^2(\Omega_\epsilon)} + D_i \int_{\Omega_\epsilon} |\nabla c_{i,\epsilon}|^2   \,dx + \frac{1}{2} \epsilon ^{\beta} D_i z_i \int_{\Omega_\epsilon} \nabla \phi_{\epsilon} \nabla (c_{i,\epsilon}^2) \,dx \leq 0.
\end{equation}
Since $c_{i,\epsilon}(t) \in H^1(\Omega_\epsilon) \cap L^\infty (\Omega_\epsilon)$, we have $c^2_{i,\epsilon}(t) \in H^1(\Omega_\epsilon)$. We test Poisson's equation (\ref{weak_Poisson}) with $c^2_{i,\epsilon}(t)$ to obtain  
\begin{equation}\label{energ_est_4}
 \epsilon^\alpha \int_{\Omega_\epsilon} \nabla \phi_\epsilon  \nabla(c^2_{i,\epsilon}) \,dx = \int_{\Omega_\epsilon} c^2_{i,\epsilon} \sum_{i=1}^{P} z_i c_{i,\epsilon}  \,dx + \int_{\partial \Omega_\epsilon } \xi_\epsilon c^2_{i,\epsilon}\,dS(x).
 \end{equation}
 Now using (\ref{energ_est_4}) in (\ref{energ_est_3}) we get
 \begin{eqnarray}\label{energ_est_5}
  &&\frac{1}{2} \frac{d}{dt} \lVert c_{i,\epsilon} \rVert ^2_{L^2(\Omega_\epsilon)} + D_i \int_{\Omega_\epsilon} |\nabla c_{i,\epsilon}|^2   \,dx \nonumber\\
  &&\leq \frac{1}{2}\epsilon^{-\alpha+\beta} D_i  \left | z_i {\int_{\Omega_\epsilon} c^2_{i,\epsilon} \sum_{i=1}^{P} z_i c_{i,\epsilon}  \,dx} \right | + \frac{1}{2} \epsilon^{-\alpha+\beta} D_i  \left | z_i \int_{\partial \Omega_\epsilon } \xi_\epsilon c^2_{i,\epsilon}\,dS(x) \right |.\nonumber\\
 \end{eqnarray}
Let $C^\prime =\displaystyle \max_{1\leq i \leq P} |z_i|$. Then
\begin{eqnarray}\label{energ_est_6}
&&\left | {\int_{\Omega_\epsilon} c^2_{i,\epsilon} \sum_{i=1}^{P} z_i c_{i,\epsilon}  \,dx} \right | \leq  C^ \prime {\int_{\Omega_\epsilon} c^2_{i,\epsilon} \sum_{i=1}^{P}  c_{i,\epsilon}  \,dx} \leq  C^\prime\int_{\Omega_\epsilon} \left (\sum_{i=1}^{P} c_{i,\epsilon}\right )^3 \,dx\nonumber \\
&& \leq C^ \prime \left (\sum_{i=1}^P \lVert c_{i,\epsilon}\rVert _{L^3(\Omega_\epsilon)}\right )^3 \leq C, 
\end{eqnarray}
for some $C>0$ independent of $\epsilon$ and $t$. The last inequality in (\ref{energ_est_6}) follows from (\ref{EstLinftyLp}). So
\begin{equation}\label{energ_est_7}
  \frac{1}{2}\epsilon^{-\alpha+\beta} D_i  \left | z_i {\int_{\Omega_\epsilon} c^2_{i,\epsilon} \sum_{i=1}^{P} z_i c_{i,\epsilon}  \,dx} \right  | \leq \epsilon^{-\alpha+\beta} C.
\end{equation}
Again,
\begin{eqnarray}\label{energ_est_8}
 \left | \int_{\partial \Omega_\epsilon } \xi_\epsilon c^2_{i,\epsilon}\,dS(x) \right |  \leq \xi^* \epsilon \int_{\Gamma_\epsilon} c^2_{i,\epsilon} \,dS(x) + \xi^* \int_{\partial \Omega} c^2_{i,\epsilon} \,dS (x).
\end{eqnarray}
Using the scaled trace-inequality from Lemma \ref{TraceInequalityLemma} we get
\begin{equation}\label{energ_est_9}
 \int_{\Gamma_\epsilon} c^2_{i,\epsilon}dS(x) \leq C_1 \left ( \frac{1}{\epsilon} \int_{\Omega_\epsilon} c^2_{i,\epsilon} \,dx + \epsilon \int_{\Omega_\epsilon} |\nabla c_{i,\epsilon}|^2 \,dx \right).
\end{equation}
Next, we estimate the last term of (\ref{energ_est_8}) using the extension operator from Lemma \ref{lemma_extension} and the weighted trace inequality (see \cite[p. 63, Excercise II.4.1]{Gal11}) to obtain
\begin{eqnarray}\label{energ_est_10}
 \int_{\partial \Omega} c^2_{i,\epsilon} \,dS (x) =\int_{\partial \Omega} \widetilde{c_{i,\epsilon}}^2 \,dS (x) 
 \leq C_2 (\Omega,\delta)  \int_{\Omega} \widetilde{c_{i,\epsilon}}^2 \,dx + \delta \int_\Omega |\nabla \widetilde{c_{i,\epsilon}}|^2 \,dx \nonumber \\
  \leq C_3 \int_{\Omega_\epsilon} c^2_{i,\epsilon} \,dx + C_4\delta \int_{\Omega_\epsilon} |\nabla c_{i,\epsilon}|^2 \,dx 
\end{eqnarray}
for any $\delta >0$. Finally, utilizing (\ref{energ_est_10}), (\ref{energ_est_9}), (\ref{energ_est_7}) in (\ref{energ_est_5}), we get
\begin{eqnarray}\label{energ_est_11}
 &&\frac{1}{2} \frac{d}{dt} \lVert c_{i,\epsilon} \rVert ^2_{L^2(\Omega_\epsilon)} + D_i \int_{\Omega_\epsilon} |\nabla c_{i,\epsilon}|^2   \,dx \nonumber\\
&& \leq C \epsilon ^{-\alpha + \beta} + \frac{1}{2} \epsilon ^{-\alpha + \beta} D_i | z_i | \xi^* (C_1 \epsilon ^2 + C_4 \delta ) \int_{\Omega_\epsilon} |\nabla c_{i,\epsilon}|^2 \,dx \nonumber \\
&&+ \frac{1}{2} \epsilon ^{-\alpha + \beta} D_i | z_i | \xi^* (C_1  + C_3 ) \int_{\Omega_\epsilon}  c^2_{i,\epsilon} \,dx.
\end{eqnarray}
Because of (\ref{EstLinftyLp}), we have that the last term of (\ref{energ_est_11}) is bounded by $ C \epsilon^{-\alpha+ \beta}$, for some $C>0$ independent of $\epsilon$ and $t$. Again, for $z_i \neq 0$, if we choose $\epsilon$ and $\delta$ small enough so that 
\begin{eqnarray*}
 C_1 \epsilon^2 < \frac{1}{4 |z_i| \xi^*} , \ \ \ C_4 \delta < \frac{1}{4 |z_i|\xi^*},
\end{eqnarray*}
then we can absorb the gradient term by the left hand side. These arguments lead to
\begin{equation}\label{energ_est_12}
  \frac{d}{dt} \lVert c_{i,\epsilon} \rVert ^2_{L^2(\Omega_\epsilon)} + D_i \int_{\Omega_\epsilon} |\nabla c_{i,\epsilon}|^2   \,dx \leq  C \epsilon ^{-\alpha+ \beta},
\end{equation}
for some $C>0$ independent of $\epsilon$ and $t$. Integration with respect to time gives $\eqref{Thm_energy_est_2}$.
\end{proof}
Now, we are able to prove the uniform $L^\infty$-estimate for the microscopic solutions $c_{i,\epsilon}$ for $ i=1, \ldots P$. Our results are based on \cite[Theorem 6.1 and Remark 6.2]{Lad68}. An important aspect in the proof is the estimate of the drift term, where refined arguments are required in order to avoid the occurrence of  higher (than $H^{1}$) norms of the potential. 
\begin{theorem} \label{Thm_L_infty_est}
There exists a constant $C >0$ which is independent of $\epsilon$ such that
\begin{equation*}
\lVert c_{i, \epsilon} \rVert_{L^\infty((0,T) \times \Omega_\epsilon)} \leq C.
\end{equation*}
\end{theorem}
\begin{proof}
For $ u \in L^2 (0,T ; H^1(\Omega))$ with $ \partial_t u \in L^2(0,T;H^1 (\Omega)^\prime)$, where $\Omega$ is an arbitrary bounded open set, we define $ u ^+  (t,x): = \max 
\{u(t,x),0\}$. Then we have, see, e.g., \cite{Wac16},
\begin{equation}\label{L_infty_1}
 < \partial_t u(t), u^+(t)>_{H^1(\Omega)^\prime,H^1(\Omega)} = \frac{1}{2} \frac{d}{dt} \lVert u^+(t) \rVert^2_{L^2(\Omega)} \ \  \text{for a.e. $t \in (0,T)$.} 	
\end{equation}
Now, let $\omega \in (0, \infty)$ and $k \in [0,\infty)$ be arbitrary and set $W := e^{-\omega t} c_{i,\epsilon}$ and $W_k := (W-k)^+$. We test the equation (\ref{Exist_1}) with $e^{-\omega t}W_k$ to obtain 
\begin{eqnarray}\label{L_infty_3}
&&< \partial_t c_{i,\epsilon} , e^{-\omega t} W_k>_{H^1(\Omega_\epsilon)^\prime,H^1 (\Omega_\epsilon)}+ D_i  \int_{\Omega_\epsilon} h^\prime_p (c_{i,\epsilon}) \nabla c_{i,\epsilon} e^{-\omega t } \nabla W_k \,dx \nonumber \\
&&= - \epsilon ^\beta  D_i \int_{\Omega_\epsilon} z_i c_{i,\epsilon} \nabla \phi_\epsilon e^{-\omega t} \nabla W_k \,dx .
\end{eqnarray}
Due to 
\begin{equation*}\label{L_infty_2}
\partial_t (e^{-\omega t}c_{i, \epsilon}-k) = - \omega  e^{-\omega t}c_{i,\epsilon} + e^{-\omega t} \partial_t c_{i,\epsilon}, \ \ \text{for a.e. $t \in (0,T)$},
\end{equation*} 	
and \eqref{L_infty_1}, we have 
\begin{eqnarray}\label{L_infty_4}
	&&< \partial_t c_{i,\epsilon} , e^{-\omega t} W_k>_{H^1(\Omega_\epsilon)^\prime,H^1 (\Omega_\epsilon)} \nonumber \\
	&&= < \partial _t (e ^{-\omega t}c_{i,\epsilon}-k), W_k>_{H^1(\Omega_\epsilon)^\prime,H^1 (\Omega_\epsilon)} +\omega \int_{\Omega_\epsilon} e ^{-\omega t} c_{i,\epsilon} W_k \,dx  \nonumber \\
	&& = \frac{1}{2} \frac{d}{dt} \int_{\Omega_\epsilon} W^2_k \,dx +\omega \int_{\Omega_\epsilon} W W_k \,dx.
\end{eqnarray}
Now utilizing (\ref{L_infty_4}), $ h^\prime_p \geq 1$ in (\ref{L_infty_3}), we get
\begin{eqnarray}\label{L_infty_5}
&&\frac{1}{2} \frac{d}{dt} \int_{\Omega_\epsilon} W^2_k \,dx +D_i \lVert  \nabla W_k \rVert^2 _{L^2(\Omega_\epsilon)} \nonumber \\
&& \leq  -\epsilon ^ \beta D_i z_i \int_{\Omega_\epsilon} W \nabla W_k \nabla \phi_\epsilon \,dx - \omega \int_{\Omega_\epsilon} W W_k \,dx \nonumber \\
&& = -\epsilon ^ \beta D_i z_i \int_{\Omega_\epsilon} W \nabla W_k \nabla \phi_\epsilon \,dx - \omega \int _{\{W(t) >k\}} W^2 \,dx + \omega \int _{\{W(t) >k\}} W k \,dx \nonumber \\
&& \leq -\epsilon ^ \beta D_i z_i \int_{\Omega_\epsilon} W \nabla W_k \nabla \phi_\epsilon \,dx - \frac{\omega}{2} \int _{\{W(t) >k\}} W^2 \,dx + \frac{\omega}{2}\int _{\{W(t) >k\}} k^2 \,dx,\nonumber \\
\end{eqnarray}
where $\{W(t) >k\} := \{x \in \Omega_\epsilon : W (t,x) >k\}$.

First, we consider the first term of the right hand side of (\ref{L_infty_5}).
\begin{eqnarray} \label{L_infty_6}
&&-\epsilon ^ \beta D_i z_i \int_{\Omega_\epsilon} W \nabla W_k \nabla \phi_\epsilon \,dx  \nonumber \\
&& = -\epsilon ^ \beta D_i z_i \int_{\Omega_\epsilon} (W-k) \nabla W_k \nabla \phi_\epsilon \,dx - \epsilon^\beta D_i z_i k \int_{\Omega_\epsilon} \nabla W_k \nabla \phi_\epsilon \,dx \nonumber \\
&& = -\epsilon ^ \beta D_i z_i \int_{\Omega_\epsilon} W_k \nabla W_k \nabla \phi_\epsilon \,dx - \epsilon^\beta D_i z_i k \int_{\Omega_\epsilon} \nabla W_k \nabla \phi_\epsilon \,dx \nonumber \\
&& =\underbrace{- \frac{1}{2}\epsilon ^ \beta D_i z_i \int_{\Omega_\epsilon}  \nabla (W^2_k) \nabla \phi_\epsilon \,dx}_\text{$I_1$} \underbrace{ \ \ \ - \epsilon^\beta D_i z_i k \int_{\Omega_\epsilon} \nabla W_k \nabla \phi_\epsilon \,dx}_\text{$I_2$}.
\end{eqnarray}
Using $W_k^2$ as a test-function in the weak formulation of Poisson's-equation $\eqref{weak_Poisson}$, we obtain
\begin{align*}
I_1 = 
\underbrace{-\frac{1}{2} \epsilon^{ \beta - \alpha} D_i z_i \int_{\Omega_\epsilon } W^2_k  \sum_{i=1}^{P} z_i c_{i, \epsilon} \,dx}_\text{$A_1$} \ \ \ \underbrace{- \frac{1}{2} \epsilon ^{ \beta - \alpha} D_i z_i \int_{\partial \Omega_\epsilon } W^2_k  \xi_{\epsilon} dS (x).}_\text{$A_2$}
\end{align*}
The term $A_1$ can be estimated using Corollary \ref{Remark_log_bound} by
\begin{eqnarray*}
A_1 \leq \frac{1}{2} \epsilon ^{ \beta-\alpha} D_i |z_i | \lVert W_k \rVert ^2 _{L^4(\Omega_\epsilon)} \left \lVert  \sum_{i=1}^{P} z_i c_{i, \epsilon} \right \rVert _{L^2 (\Omega_\epsilon)} \leq C \epsilon ^{-\alpha + \beta}  \lVert W_k \rVert ^2 _{L^4(\Omega_\epsilon)}.
\end{eqnarray*}
Using the Gagliardo-Nirenberg-interpolation inequality (see, for example, \cite[Exercise II.3.12]{Gal11} or \cite{NirenbergEllipticEquations}) and Lemma  \ref{lemma_extension}, for any $ \delta_1 >0$, we obtain 
\begin{eqnarray}\label{L_infty_7}
A_1 &&\leq C \epsilon ^ {- \alpha + \beta} \lVert \widetilde{W_k} \rVert ^2 _{L^4(\Omega)} \nonumber \\
&& \leq C \epsilon ^ {- \alpha + \beta} \left (\delta_1 \lVert  \nabla \widetilde{W_k} \rVert ^2 _{L^2(\Omega)} + C (\delta_1, \Omega) \lVert \widetilde{W_k} \rVert ^2 _{L^2(\Omega)} \right ) \nonumber \\
&& \leq C \epsilon ^ {- \alpha + \beta}  \left (C_1 \delta_1 \lVert  \nabla W_k \rVert ^2 _{L^2(\Omega_\epsilon)}  +C_2(\delta_1) \lVert   W_k \rVert ^2 _{L^2(\Omega_\epsilon)}\right ).
\end{eqnarray}
Furthermore, for the term $A_2$ we get
\begin{eqnarray*}
A_2 
&& \leq C \xi^*  \epsilon ^ {-\alpha + \beta} \int_{\Gamma_\epsilon} \epsilon W^2_k \,dS (x)   + C \xi ^*  \epsilon ^ {-\alpha + \beta}  \int_{\partial \Omega} W^2_k \,d S (x)  \nonumber\\
&& \leq C \epsilon ^ {-\alpha + \beta } \left ( \int_{\Omega_\epsilon} W^2_k \,dx + \epsilon ^2   \lVert \nabla W_k \rVert^2_{L^2(\Omega_\epsilon)} \right ) +C \epsilon ^ {-\alpha + \beta } \int_{\partial \Omega} W_k^2;
\end{eqnarray*}
here in the last inequality we used again Lemma \ref{TraceInequalityLemma}. Consequently, by means of \cite[p. 63, Exercise II.4.1]{Gal11} and Lemma \ref{lemma_extension}, for any $\delta_2 >0$ we get   
\begin{eqnarray} \label{L_infty_8}
A_2 &&\leq C \epsilon ^ {-\alpha + \beta } \left ( \int_{\Omega_\epsilon} W^2_k \,dx + \epsilon ^2 \lVert \nabla W_k \rVert^2_{L^2(\Omega_\epsilon)} \right ) \nonumber \\
&& +C \epsilon ^ {-\alpha + \beta } \left (  \delta _2 \lVert \nabla W_k \rVert^2_{L^2(\Omega_\epsilon)}+ C_3(\delta_2) \int_{\Omega_\epsilon}  W^2_k  \,dx \right).
\end{eqnarray}
So (\ref{L_infty_7}), (\ref{L_infty_8}) imply 
\begin{eqnarray*} 
I_1 && \leq C \epsilon ^ {- \alpha + \beta } (\epsilon^2 + C_1 \delta_1 + \delta_2) \lVert \nabla W_k \rVert^2_{L^2(\Omega_\epsilon)} \nonumber\\
&&+  C \epsilon ^ {- \alpha + \beta }\left(C_2(\delta_1)+ C_3(\delta_2)+1\right) \int_{\Omega_\epsilon}  W^2_k  \,dx.
\end{eqnarray*}
Then, the inequality
\begin{equation*}
\int_{\Omega_\epsilon} W^2_k \,dx \leq 2 \left( \int_ {\{W(t) >k\}} W^2 \,dx + k^2 \int_ {\{W(t) >k\}} \,dx \right)
\end{equation*}
leads to 
\begin{eqnarray}\label{L_infty_9}
I_1 &&\leq C \epsilon ^ {- \alpha + \beta } (\epsilon^2 + C_1 \delta_1 + \delta_2) \lVert \nabla W_k \rVert^2_{L^2(\Omega_\epsilon)} \nonumber \\
&& +   C  \epsilon ^ {- \alpha + \beta } C_4 (\delta_1, \delta_2) \left( \int_ {\{W(t) >k\}} W^2 \,dx + k^2 \int_ {\{W(t) >k\}} \,dx \right).
\end{eqnarray}

To estimate $I_2$ from (\ref{L_infty_6}) we use $kW_k$ as a test-function in  Poisson's-equation $\eqref{weak_Poisson}$ to obtain
\begin{eqnarray}\label{L_infty_10}
I_2 = \underbrace{ -  \epsilon^{ \beta - \alpha} D_i z_i \int_{\Omega_\epsilon } k W_k \sum_{i=1}^{P} z_i c_{i, \epsilon} \,dx }_\text{$A_3$} \ \ \underbrace{-  \epsilon ^{ \beta - \alpha} D_i z_i \int_{\partial \Omega_\epsilon } k W_k  \xi_{\epsilon} dS (x).}_\text{$A_4$}
\end{eqnarray}
Using again Corollary \ref{Remark_log_bound} (remember $p \geq 4$) we immediately get
\begin{eqnarray}\label{L_infty_11}
A_3 
 \leq  C \epsilon^{-\alpha +\beta} \lVert k \rVert_{L^2(\{W(t)>k\})} \lVert W_k \rVert_{L^4(\Omega_\epsilon) }.
\end{eqnarray}
Now, using similar arguments used in the estimation of $A_1$, we have from (\ref{L_infty_11}) for any $\delta_3>0$,  
\begin{eqnarray*}
 A_3 \leq C \epsilon ^ {- \alpha + \beta} \left(  \int_{\{W(t)>k\}} k^2 \,dx + C_5 \delta_3  \lVert \nabla W_k \rVert^2_{L^2(\Omega_\epsilon)} + C_6 (\delta_3) \int_{\{W(t)>k\}}  W^2_k  \,dx  \right).
\end{eqnarray*}
This gives
\begin{eqnarray} \label{L_infty_12}
 A_3 &&\leq C \epsilon ^{-\alpha + \beta }C_5\delta_3  \lVert \nabla W_k \rVert^2_{L^2(\Omega_\epsilon)}\nonumber \\
 &&+C\epsilon ^{-\alpha + \beta}C_7 (\delta_3) \left( \int_{\{W(t)>k\}} W^2 \,dx + \int_{\{W(t)>k\}} k^2 \,dx \right ).
\end{eqnarray}
Now let us estimate $A_4$ from (\ref{L_infty_10}). We have
\begin{eqnarray}\label{L_infty_13}
 A_4 \leq C \xi ^* \epsilon^{-\alpha + \beta} \int_{\Gamma_\epsilon} \epsilon  k |W_k | \,dS(x) + C \xi ^* \epsilon^{-\alpha + \beta} \int_{\partial \Omega}  k |W_k | \,dS(x).
\end{eqnarray}
The scaled trace-inequality from Lemma \ref{TraceInequalityLemma} yields 
\begin{equation*}
 \int_{\Gamma_\epsilon} k |W_k| \,dS(x) \leq C \left ( \frac{1}{\epsilon}\int_{\Omega_\epsilon } k |W_k | \,dx +   \lVert k \nabla W_k \rVert _{L^1(\Omega_\epsilon)} \right ).
\end{equation*}
Also, \cite[p. 63, Exercise II.4.1]{Gal11} and Lemma \ref{lemma_extension} lead to 
\begin{equation*}
 \int_{\partial \Omega} k |W_k| \,dx \leq C_8 (\delta_4) \int_{\Omega_\epsilon} k |W_k| \,dx 
 + C_9 \delta_4 \lVert k \nabla W_k \rVert _{L^1(\Omega_\epsilon)},
\end{equation*}
for any $\delta_4 >0$.
Making use of these two inequalities in (\ref{L_infty_13}), we obtain
\begin{eqnarray*}
 A_4 \leq C \epsilon ^{-\alpha + \beta} (\epsilon +C_9  \delta_4) \lVert k \nabla W_k \rVert _{L^1(\Omega_\epsilon)} + C \epsilon ^{-\alpha + \beta}(1+C_8 (\delta_4)) \int_{ \Omega_\epsilon} k |W_k| \,dx.
\end{eqnarray*}
Hence,
\begin{eqnarray}\label{L_infty_14}
 A_4 &&\leq C \epsilon ^{-\alpha + \beta} (\epsilon + C_{9} \delta_4) \lVert  \nabla W_k \rVert^2 _{L^2(\Omega_\epsilon)} \nonumber \\
 &&+  C \epsilon ^{-\alpha + \beta} C_{10}(\delta_4) \left ( \int_{\{W(t)>k\}} W^2 \,dx + \int_{\{W(t)>k\}} k^2 \,dx \right).
\end{eqnarray}
Consequently, from (\ref{L_infty_12}), (\ref{L_infty_14}), we get
\begin{eqnarray}\label{L_infty_15}
 I_2 &&\leq C \epsilon ^ {-\alpha + \beta} (\epsilon +C_5 \delta_3 +C_{9} \delta_4)\lVert  \nabla W_k \rVert^2 _{L^2(\Omega_\epsilon)} \nonumber \\
 &&+ C \epsilon ^ {-\alpha + \beta}  C_{11}(\delta_3, \delta_4) \left (\int_{\{W(t)>k\}} W^2 \,dx + \int_{\{W(t)>k\}} k^2 \,dx \right).
 \end{eqnarray}
 Now using (\ref{L_infty_15}), (\ref{L_infty_9}) in (\ref{L_infty_5}), we get 
 \begin{eqnarray*}
  &&\frac{1}{2} \frac{d}{dt} \int_{\Omega_\epsilon}
  W^2_k \,dx +D_i \lVert  \nabla W_k \rVert^2 _{L^2(\Omega_\epsilon)} \\
  &&\leq C \epsilon ^{- \alpha + \beta} (\epsilon^2 +\epsilon+C_1 \delta_1 + \delta_2 + C_5 \delta_3 + C_{9} \delta_4 ) \lVert  \nabla W_k \rVert^2 _{L^2(\Omega_\epsilon)} \\
  &&+ C \epsilon ^{- \alpha + \beta} C_{12} (\delta_1, \delta_2, \delta_3, \delta_4)  \int_{\{W(t)>k\}} W^2 \,dx  \\
  &&+\left [C \epsilon ^{- \alpha + \beta} C_{13} (\delta_1, \delta_2, \delta_3, \delta_4) + \frac{\omega}{2} \right ]\int_{\{W(t)>k\}} k^2 \,dx 
- \frac{\omega}{2}\int_{\{W(t)>k\}} W^2 \,dx.
 \end{eqnarray*}
From the above equation we see that if we choose $\epsilon, \delta_1,...,\delta_4$ small enough so that $C \epsilon ^{- \alpha + \beta} (\epsilon^2 +\epsilon+C_1 \delta_1 + \delta_2 + C_5 \delta_3 + C_{9} \delta_4 ) < \frac{D_i}{2}$, then we can absorb the gradient term in the right hand side by the gradient term in the left hand side. Note that $\delta_1,...,\delta_4$ do not depend on $\epsilon$, whenever $\epsilon < \epsilon_0$, for some $\epsilon_0 >0$. Then, if we assume $\frac{\omega}{2} > C \epsilon^{-\alpha + \beta} C_{12} (\delta_1,..., \delta_4) $, we can remove the term $\int_{\{W(t)>k\}} W^2 \,dx$ from the right hand side. Therefore, we get
\begin{eqnarray}\label{L_infty_16}
 \frac{d}{dt} \int_{\Omega_\epsilon} W^2_k \,dx +  D_i \lVert  \nabla W_k \rVert^2 _{L^2(\Omega_\epsilon)}\leq (C \epsilon^{-\alpha + \beta} + \omega) \int_{\{W(t)>k\}} k^2 \,dx.
\end{eqnarray}
Choose $k \geq \lVert c^0_i\rVert_{L^\infty(\Omega)} + 1$, which implies $W_k(0) =0$. 
Now, we integrate (\ref{L_infty_16}) with respect to $t$ to obtain 
\begin{eqnarray*}
 \displaystyle \max_{t \in [0,T]} \int_{\Omega_\epsilon} W^2_k (t) \,dx + D_i \lVert  \nabla W_k \rVert^2 _{L^2((0,T) \times \Omega_\epsilon)} \leq (C \epsilon^{-\alpha + \beta} + \omega) k^2 \int_{\{W >k\}} \,dx \,dt,
\end{eqnarray*}
where $\{W > k\} := \{(t,x) \in (0,T) \times  \Omega_\epsilon : W(t,x) >k\}$.
Now utilizing \cite[p. 102, Theorem 6.1 and p. 103, Remark 6.2]{Lad68} with $r =q = \frac{2(n+2)}{n}, \ \kappa= \frac{2}{n} $, it follows that
\begin{equation}\label{L_infty_17}
 \esssup _{(t,x) \in (0,T) \times \Omega_\epsilon} W (t, x ) \leq C_{14} \left[1 + C_{15}(C_{16} \epsilon ^{-\alpha + \beta}+ \omega) ^{\frac{n+2}{4}}\right].
\end{equation}
We emphasize that the constants $C_{14}, C_{15}, C_{16}$ are independent of $\epsilon$. In fact, the proof of \cite[Theorem 6.1]{Lad68} is based on the embedding 
\begin{align*}
L^{\infty}((0,T),L^2(\Omega_{\epsilon}))\cap L^2((0,T),H^1(\Omega_{\epsilon})) \hookrightarrow L^r((0,T)\times \Omega_{\epsilon}).
\end{align*}
The embedding constant is independent of $\epsilon$, since due to the extension 
operator from Lemma \ref{lemma_extension} we easily obtain for any $u_\epsilon \in L^{\infty}((0,T),L^2(\Omega_{\epsilon}))\cap L^2((0,T),H^1(\Omega_{\epsilon}))$ that
\begin{align*}
 \left ( \int_0^T \int_{\Omega_\epsilon} |u_\epsilon|^r  \right )^{\frac{1}{r}} \leq c \left [  \esssup _{t \in (0,T)}\left ( \int_{\Omega_\epsilon} |u_\epsilon (t,x)|^2 \,dx \right)^{\frac{1}{2}} + \lVert \nabla u_\epsilon \rVert_{L^2((0,T)\times \Omega_\epsilon)} \right ],
\end{align*}
with a constant $c$ independent of $\epsilon$.
\\
 Consequently, from (\ref{L_infty_17}) we conclude that $W$ is bounded from above by some positive constant $C$ in $(0,T) \times \Omega_\epsilon$ uniformly with respect to $\epsilon$. Since $c_{i,\epsilon}$ are non-negative, the proof is complete.
\\
\end{proof}
\begin{proposition}\label{thm_time_der}
There exists a constant $C>0$ independent of $\epsilon$ such that 
\begin{equation}
\lVert \partial_{t}c_{i,\epsilon}\rVert_{L^2(0,T;H^1(\Omega_\epsilon)^\prime)} \leq C. 
\end{equation}
\end{proposition}
\begin{proof}
 Let $\psi \in L^2(0,T;H^1(\Omega_\epsilon))$. From (\ref{Exist_1}) we have
 \begin{eqnarray*}
 && \int_0^T < \partial_{t} c_{i,\epsilon}, \psi>_{H^1(\Omega_\epsilon)^\prime, H^1(\Omega_\epsilon)} \,dt\\
 && \leq D_i \int_0^T  \int_{\Omega_\epsilon} (1+ \eta c_{i,\epsilon}^{p-1}) | \nabla c_{i,\epsilon}| |\nabla \psi| \,dx \,dt + \epsilon^\beta D_i |z_i | \int_0^T  \int_{\Omega_\epsilon} c_{i,\epsilon} |\nabla \phi_\epsilon| |\nabla \psi| \,dx \,dt \\
 && \leq C \lVert \nabla c_{i,\epsilon} \rVert_{L^2(0,T;L^2(\Omega_\epsilon))} \lVert \psi \rVert_{L^2(0,T;H^1(\Omega_\epsilon))}+C \epsilon^\beta \lVert \nabla \phi_\epsilon \rVert_{L^2(0,T;L^2(\Omega_\epsilon))} \lVert \psi \rVert_{L^2(0,T;H^1(\Omega_\epsilon))}\\
 &&  \leq C (1+ \epsilon^{-\alpha+\beta} )\lVert \psi \rVert_{L^2(0,T;H^1(\Omega_\epsilon))};
 \end{eqnarray*}
here we used Proposition \ref{Thm_energ_est_1} and Theorem  \ref{Thm_L_infty_est}. This completes the proof.
\end{proof}


\section{Derivation of macroscopic model}
\label{Deriv_macro_model}

In this section, we derive the macroscopic (homogenized) drift-diffusion model by means of two-scale convergence concepts introduced in \cite{Ngu89, All92, Neu96}.  The two-scale convergence results are based on the uniform \textit{a priori} estimates for the microscopic solutions $c_{i,\epsilon}, i=1, \ldots,P$ and $\phi_{\epsilon}$ from Section \ref{Sect_unif_est}.
\begin{definition}[\cite{All92}]
A sequence of functions $u_\epsilon$ in $L^2( (0,T) \times \Omega)$ is said to two-scale converge to a function $u_0 \in L^2 ((0,T) \times \Omega \times Y )$, if for all $\psi \in L^2 ((0,T) \times \Omega ; C_{per}(\overline{Y}))$ the following holds:
\begin{eqnarray*}
\lim_{\epsilon \to 0} \int_{0}^{T} \int_{\Omega} u_\epsilon (t,x) \psi \left (t,x, \frac{x}{\epsilon}\right) \,dx \,dt = \int_{0}^{T} \int_{\Omega} \int_Y u_0 (t,x,y) \psi (t,x,y) \,dy \,dx \,dt .
\end{eqnarray*}
\end{definition}
The following compactness result is well known:
For every bounded sequence $u_{\epsilon} \in L^2((0,T),H^1(\Omega))$ there exists $u_0 \in L^2((0,T),H^1( \Omega))$ and $u_1 \in L^2((0,T)\times \Omega , H^1_{per}(Y)/\R)$ such that up to a subsequence
\begin{align*}
u_{\epsilon} &\rightarrow u_0 &\mbox{ in the two-scale sense,}
\\
\nabla u_{\epsilon} &\rightarrow \nabla u_0 + \nabla_y u_1 &\mbox{ in the two-scale sense}.
\end{align*}

A proof can be found in \cite{All92} (see also \cite{Ngu89}) for the time-independent case, and can be easily generalized to  time-dependent problems. Note, however, that in our problem the microscopic solutions are only defined on the perforated domain $\Omega_{\epsilon}$, and to apply the above two-scale compactness results we have to extend $c_{i,\epsilon}$ and $\phi_{\epsilon}$ in a suitable way to the whole domain $\Omega$. In a trivial way we can extend a function $f: \Omega_{\epsilon} \rightarrow \R$ by zero to $\Omega$ (this extension is denoted by $\chi_{\Omega_{\epsilon}} f$, and $\chi_{\Omega_{\epsilon}} $ denotes the characteristic function of $\Omega_{\epsilon}$). For the zero-extension, the regularity properties are in general not preserved. These are, however, the basis for obtaining strong compactness results, which allow to pass to the limit in the nonlinear terms. Hence, we extend the functions $c_{i,\epsilon}$ by using the extension operator $\widetilde{c_{i,\epsilon}}$ from Lemma \ref{lemma_extension}, pointwise in $t\in (0,T)$. It is well known, see, for example, \cite{Cio79} and \cite{Acerbi1992}, that  $\widetilde{c_{i,\epsilon}}$ fulfills the same regularity results and uniform \textit{a priori } estimates as $c_{i,\epsilon}$ in Proposition \ref{Thm_energ_est_1} and Theorem \ref{Thm_L_infty_est}. In Lemma \ref{lemma_extension}, we additionally prove that $\widetilde{c_{i,\epsilon}}$ is also essentially bounded and non-negative. These properties are needed to pass to the limit in the microscopic model, in particular in the nonlinear term $h_p(c_{i,\epsilon})$ entering the diffusion coefficients. Finally, we emphasize that we have no information about $\partial_t \widetilde{c_{i,\epsilon}}$, but this is not necessary for the derivation of the macroscopic model.  

In the following proposition, we prove convergence results for the microscopic solutions. Here, the characteristic function on $Y^f$ is denoted by $\chi_{Y^f}$.

\begin{proposition}\label{Thm_conv_micro_sol}
	\begin{enumerate}[label=(\alph*)]
	\item There exist $c_{i,0} \in L^\infty((0,T) \times \Omega) \cap L^2(0,T;H^1(\Omega))$, $c_{i,0} \geq 0$ almost everywhere in $(0,T) \times \Omega$ with $\partial_t c_{i,0} \in L^2(0,T;H^1(\Omega)^\prime)$ and $c_{i,1} \in L^2((0,T) \times \Omega; H^1_{per} (Y^f)/ \R)$ such that for every $q\in [1,\infty)$, up to a subsequence  
		\begin{align}
		& \rchi_{\Omega_\epsilon} c_{i,\epsilon} \rightarrow \rchi_{Y^f}  c_{i,0}   &&\text{in the two-scale sense}, \label{THM_CONV_1}\\
		& \rchi_{\Omega_\epsilon} \nabla c_{i,\epsilon} \rightarrow \rchi_{Y^f}  \left (\nabla c_{i,0} + \nabla _y c_{i,1}  \right )  &&\text{in the two-scale sense},\label{THM_CONV_2}\\
		& \widetilde{c_{i,\epsilon}} \rightarrow c_{i,0} &&\text{strongly in $L^q((0,T)\times \Omega)$},
		\\
		& h_p^{\prime}\left(\widetilde{c_{i,\epsilon}}\right) \rightarrow h_p^{\prime}\left(c_{i,0}\right) &&\text{strongly in } L^q((0,T)\times \Omega).
		\end{align}
	\item There exists $\phi_0 \in L^2 (0,T; H^1(\Omega))$ and $\phi_1 \in L^2 ((0,T) \times \Omega;H^1_{per}(Y^f)/ \R)$ such that, up to a subsequence  \begin{align}
	& \epsilon^\alpha \rchi_{\Omega_\epsilon} \phi_{\epsilon} \rightarrow \rchi_{Y^f} \phi_{0}   &&\text{in the two-scale sense}, \label{THM_CONV_3}\\
	& \epsilon^\alpha \rchi_{\Omega_\epsilon} \nabla \phi_{\epsilon} \rightarrow \rchi_{Y^f} \left (\nabla \phi_{0} + \nabla _y \phi_{1}  \right )  &&\text{in the two-scale sense}\label{THM_CONV_4}. 
	\end{align}	
\end{enumerate}
\end{proposition}	
\begin{proof}
The two-scale convergence results are quite standard (see, for example, \cite{All92} for more details) and follow from the \textit{a priori } estimates in Proposition \ref{Thm_energ_est_1}. 
It remains to prove the strong convergence results for the extension. 
Using again the \textit{a priori } estimates for $c_{i,\epsilon}$, we obtain from  \cite[Lemma 9 and Lemma 10]{Gah16} the strong convergence of $\widetilde{c_{i,\epsilon}}$ to $c_{i,0}$ in $L^2((0,T)\times \Omega)$ (up to a subsequence). This implies, again up to a subsequence, pointwise almost everywhere convergence in $(0,T) \times \Omega$. Together with the $L^{\infty}$-estimate from Theorem \ref{Thm_L_infty_est} and the dominated convergence theorem we obtain the convergence of $\widetilde{c_{i,\epsilon}}$ in $L^q((0,T)\times \Omega)$ for every $q \in [1,\infty)$. The lower semi-continuity of the norm implies $c_{i,0} \in L^{\infty}((0,T)\times \Omega)$, see also \cite[Exercise 4.6]{Bre11}. For the convergence of $h_p^{\prime}\left(\widetilde{c_{i,\epsilon}}\right)$ we use the local Lipschitz continuity of $h_p^{\prime}$ and again the essential boundedness of $ \widetilde{c_{i,\epsilon}}$, which implies together with Lemma \ref{lemma_extension}
\begin{align*}
\left\Vert h_p^{\prime}\left(\widetilde{c_{i,\epsilon}}\right) - h_p^{\prime}(c_{i,0})\right\Vert_{L^q((0,T)\times \Omega)} \le  C \Vert \widetilde{c_{i,\epsilon}} - c_{i,0}\Vert_{L^q((0,T)\times \Omega)} \rightarrow 0
\end{align*}
for $\epsilon \to 0$. This finishes the proof.
\end{proof}
We now have all ingredients to pass to the limit $\epsilon \to 0$ in the variational formulation \eqref{Exist_1}-\eqref{weak_Poisson} and to obtain the macroscopic models for different values of the parameters $\alpha$ and $\beta$.
\begin{theorem}\label{Thm_macro_model}
\begin{enumerate} [label=(\alph*)]    	    	
		\item Let $\alpha = \beta$. The limit functions $c_{i,0}$ and $\phi_0$ from Proposition \ref{Thm_conv_micro_sol} are weak solutions of the following homogenized model:
		\begin{equation}\label{Macro_eq_1}
		\begin{aligned}
		&& \partial_t c_{i,0} - \nabla_x \cdot \left[ D_i A_{hom} \nabla h_p(c_{i,0}) + D_i z_i c_{i,0}A_{hom} \nabla \phi_0 \right  ] &=0  &\text{in $(0,T) \times \Omega$}, \\
		&& -\left[ D_i A_{hom} \nabla h_p(c_{i,0}) + D_i z_i c_{i,0}A_{hom} \nabla \phi_0 \right  ] \cdot \nu &=0   &\text{on $(0,T) \times \partial \Omega$},\\
		&&c_{i,0}(0,x) &= c_i^0 (x)  &\text{in $\Omega$},
		\end{aligned}
		\end{equation}
		for all $i \in \{1,...,P\}$, and
		\begin{equation}\label{Macro_eq_2}
		\begin{aligned}
		&&-\nabla_x \cdot \left [A_{hom} \nabla \phi_0 (t,x) \right ]&= \sum_{k=1}^{P} z_i c_{i,0} (t,x) + \frac{1}{\left | Y^f \right | } \int_{\Gamma} \xi_1 (x,y) \,dy   &\text{in $(0,T) \times \Omega$}, \\
		&& A_{hom} \nabla \phi_0 (t,x)  \cdot \nu &= \frac{1}{\left | Y^f \right | } \xi_2 (x) &\text{on $(0,T) \times  \partial \Omega$}.
		\end{aligned}
		\end{equation}
		Here $A_{hom}$ is an $n \times n$ matrix given by
		\begin{equation}\label{Macro_eq_2a_0}
		A_{hom} e_k = \frac{1}{\left | Y^f \right | } \int_{Y^f} \left (\nabla_y w_k + e_k \right) \,dy, \ \ \text{for $1\leq k \leq n$,} 
		\end{equation}
where $e_k$ denotes the standard basis vector of $\mathbb{R}^n$ and 
$w_k \in H^1_{per} (Y^f)/ \mathbb{R}$ is the unique weak solution of the so-called cell problem
\begin{equation}\label{Macro_eq_2a}
\begin{aligned}
-\nabla_y \cdot \left ( \nabla_y w_k (y) +e_k\right) &=0 \ \ &\text{in $Y^f$},\\
\left ( \nabla_y w_k (y) +e_k\right) \cdot \nu  &=0  \ \ &\text{on $\partial Y^f$},\\
y \mapsto w_k (y) \ \ &\text{is $Y$-periodic.}
\end{aligned}
\end{equation}

\item If $\alpha < \beta$, then the homogenized Poisson's equation (\ref{Macro_eq_2}) remains the same, however the advective term $D_i z_i c_{i,0}A_{hom} \nabla \phi_0$ in the homogenized transport equation (\ref{Macro_eq_1}) does not appear. That is, in that case (\ref{Macro_eq_1}) becomes 
\begin{equation}\label{Macro_eq_3}
\begin{aligned}
&&  \partial_t c_{i,0} - \nabla_x \cdot  D_i A_{hom} \nabla h_p(c_{i,0})  &=0  &\text{in $(0,T) \times \Omega$}, \\
&& -D_i A_{hom} \nabla h_p(c_{i,0})  \cdot \nu &=0   &\text{on $(0,T) \times \partial \Omega$},\\
&&c_{i,0}(0,x) &= c_i^0 (x)  &\text{in $\Omega$}.
\end{aligned}
\end{equation}
\end{enumerate}
\end{theorem}
\begin{proof}
We start by proving statement \textit{(a)}. Let $\psi_0 (t,x) \in C^\infty([0,T] \times \overline{\Omega})$ with $\psi_0(T,.)=0$ and $\psi_1(t,x,y) \in C^\infty_0((0,T) \times \Omega \times \overline{Y})$, which is $Y$-periodic in $y$. Let us consider a subsequence of $\epsilon$, still denoted by $\epsilon$, along which the convergence results given in Proposition \ref{Thm_conv_micro_sol} hold. Now considering $\psi_{\epsilon}(t,x):= \psi_0(t,x) + \epsilon \psi_1 \left (t, x, \frac{x}{\epsilon}\right)$ as a test function in (\ref{Exist_1}), we get
\begin{eqnarray}\label{macro_1}
&&-\int_0^T \int_{\Omega_\epsilon} c_{i, \epsilon} \left [ \partial_t \psi_0 (t,x) + \epsilon \partial_t \psi_1 \left (t, x , \frac{x}{\epsilon}\right) \right ] \,dx \,dt \nonumber \\
&&+ D_i \int_{0}^{T} \int_{\Omega_\epsilon}  h_p^{\prime}(c_{i,\epsilon})\nabla c_{i, \epsilon} \left[ \nabla \psi_0 (t,x) + \epsilon \nabla_x \psi_1 \left(t,x,\frac{x}{\epsilon}\right) + \nabla_y \psi_1 \left(t,x,\frac{x}{\epsilon}\right) \right] \,dx \,dt \nonumber \\
&&+ \epsilon ^ \beta D_i \int_{0}^{T} \int_{\Omega_\epsilon} z_i c_{i, \epsilon} \nabla \phi_\epsilon \left[ \nabla \psi_0 (t,x) + \epsilon \nabla_x \psi_1 \left(t,x,\frac{x}{\epsilon}\right) + \nabla_y \psi_1 \left(t,x,\frac{x}{\epsilon}\right) \right] \,dx \,dt \nonumber \\
&& = \int_{\Omega_\epsilon} c^0_i (x) \psi_0 (0,x)  \,dx.
\end{eqnarray}
Let us first pass to the limit in the second term in the left-hand side of (\ref{macro_1}). Passing to the limit in the other terms will follow from simpler arguments. We have,
\begin{align*}
D_i \int_0^T \int_{\Omega_{\epsilon}}  h_p^{\prime}(c_{i,\epsilon}) \nabla c_{i,\epsilon} \nabla \psi_{\epsilon} dx =& D_i \int_0^T \int_{\Omega_{\epsilon}}   \left[ h_p^{\prime}(c_{i,\epsilon})-  h_p^{\prime}(c_{i,0})\right] \nabla c_{i,\epsilon} \nabla \psi_{\epsilon} dx 
\\
&+ D_i \int_0^T \int_{\Omega_{\epsilon}} h_p^{\prime}(c_{i,0})\nabla c_{i,\epsilon} \nabla \psi_{\epsilon} dx =: A_{\epsilon}^1 + A_{\epsilon}^2
\end{align*}
For the first term we use the strong convergence of $h_p^{\prime}\left(\widetilde{c_{i,\epsilon}}\right)$ from Proposition \ref{Thm_conv_micro_sol} and the boundedness of $\nabla c_{i,\epsilon}$ from Proposition \ref{Thm_energ_est_1} (we emphasize that $\nabla \psi_{\epsilon}$ is essentially bounded):
\begin{align*}
\vert A_{\epsilon}^1 \vert \le C \left\Vert h_p^{\prime}\left(\widetilde{c_{i,\epsilon}}\right) - h_p^{\prime}(c_{i,0})\right\Vert_{L^2((0,T)\times \Omega)} \overset{\epsilon \to 0}{\longrightarrow} 0.
\end{align*}
For the term $A_{\epsilon}^2$ we notice that $h^{\prime}_p(c_{i,0}) \nabla \psi_{\epsilon}$ is an admissible test-function in the two-scale sense. Hence, we can use the two-scale convergence result for $\chi_{\Omega_{\epsilon}} \nabla c_{i,\epsilon}$ from  Proposition \ref{Thm_conv_micro_sol} to obtain 
\begin{align*}
\lim_{\epsilon \to 0 } A_{\epsilon}^2 = D_i \int_0^T \int_{\Omega} \int_{Y^f} h^{\prime}_p(c_{i,0})\left(\nabla c_{i,0} + \nabla_y c_{i,1}\right) \left(\nabla \psi_0 (t,x)+ \nabla_y \psi_1 (t,x,y)\right) \,dy \,dx \,dt.
\end{align*}
The other terms in $\eqref{macro_1}$ can be treated in a similar way and we obtain for $\epsilon \to 0$
\begin{eqnarray}\label{macro_3}
&&-\left | Y^f \right | \int_0^T \int_{\Omega} c_{i,0} \partial_t \psi_0 \,dx \,dt 
\\
&&+D_i   \int_0^T \int_{\Omega} \int_{Y^f}  h_p^{\prime}( c_{i,0})\left(\nabla c_{i,0} + \nabla_y c_{i,1}\right) \left(\nabla \psi_0 + \nabla_y \psi_1 \right) \,dy \,dx \,dt \nonumber \\
&& +D_i z_i \int_0^T \int_{\Omega} \int_{Y^f}   c_{i,0} \left(\nabla \phi_0 + \nabla_y \phi_1 \right) \left(\nabla \psi_0 + \nabla_y \psi_1 \right)  \,dy \,dx \,dt  \nonumber \\
&& = \left | Y^f \right |  \int_{\Omega} c^0_i \psi_0(0,x) \,dx.
\end{eqnarray}
Choosing $\psi_1=0$ we obtain
that $c_{i,0}(0)=c_i^0$. Integrating by parts in time in $\eqref{macro_3}$ and using a density argument, we obtain for all $\psi_0 \in L^2(0,T; H^1(\Omega))$ and $ \psi_{1} \in L^2((0,T)  \times \Omega ; H^1_{per}(Y^f)/\R)$ that
\begin{align}
\begin{aligned}\label{macro_4}
&\left | Y^f \right | \int_0^T \left< \partial_t c_{i,0},  \psi_0 \right>_{H^1(\Omega)^\prime, H^1(\Omega)} \,dt \\
&+D_i \int_0^T \int_{\Omega} \int_{Y^f}   h_p^{\prime}(c_{i,0})\left(\nabla c_{i,0} + \nabla_y c_{i,1}\right) \left(\nabla \psi_0 + \nabla_y \psi_1 \right) \,dy \,dx \,dt  \\
& +D_i z_i \int_0^T \int_{\Omega} \int_{Y^f}   c_{i,0} \left(\nabla \phi_0 + \nabla_y \phi_1 \right) \left(\nabla \psi_0 + \nabla_y \psi_1 \right)  \,dy \,dx \,dt =0. 
\end{aligned}
\end{align}
(\ref{macro_4}) is the so-called two-scale homogenized model for the transport equation. 

Now considering $\psi_{\epsilon}(t,x):= \psi_0(t,x) + \epsilon \psi_1 \left (t, x, \frac{x}{\epsilon}\right)$ as a test function in \eqref{weak_Poisson}, by similar arguments (see \cite{Neu96} for the convergence of the integral over the microscopic boundary $\Gamma_\epsilon$), we obtain the two-scale homogenized Poisson's equation given as follows:
For all $\psi_0 \in L^2((0,T);H^1(\Omega))$ and $\psi_1 \in L^2((0,T) \times \Omega ; H^1_{per}(Y^f)/\R)$ it holds that
\begin{eqnarray}\label{macro_6}
&&\int_0^T \int_{\Omega} \int_{Y^f} \left ( \nabla \phi_0 + \nabla_y \phi_1 \right) \left( \nabla \psi_0 + \nabla _y \psi_1\right) \,dy \,dx \,dt\nonumber \\
&& = \left | Y^f \right | \int_{0}^{T} \int_{\Omega} \sum_{i=1}^{P} z_i c_{i,0} \psi_0 \,dx \,dt + \int_{0}^{T} \int_{\Omega} \int_{\Gamma} \xi_1 (x,y) \psi_0 (t,x) \,dS(y) \,dx \,dt  \nonumber \\
&&  + \int_{0}^{T} \int_{\partial \Omega} \xi_2 (x) \psi_0 (t,x) \,dS(x) \,dt.  
\end{eqnarray}
Choosing $\psi_0=0$ in (\ref{macro_6}), we get
\begin{equation}\label{macro_7}
\int_{0}^{T} \int_\Omega \int_{Y^f} [ \nabla \phi_0 + \nabla_{y} \phi_1] \cdot \nabla_{y}\psi_1 \,dy \,dx \,dt = 0. 
\end{equation}
It is well-known, see, for example, \cite{All92}, that for a given $\phi_0 \in L^2((0, T), H^1(\Omega))$ this problem admits an unique weak solution $\phi_1 \in L^2((0,T)\times \Omega; H_{per}^1(Y^f)/ \R)$ which has the representation
\begin{equation}\label{macro_7a}
\phi_{1}(t,x,y) = \displaystyle \sum_{k=1}^n  \partial_{x_k} \phi_{0}(t,x) w_k(y),
\end{equation}
 where $w_k$ is the unique weak solution of (\ref{Macro_eq_2a}). Now, choosing $\psi_0 = 0$ in the two-scale homogenized transport equation $\eqref{macro_4}$ and using the representation for $\phi_1$, we obtain with similar arguments as above (we emphasize  that $h^{\prime}_p(c_{i,0})\geq 1$)
\begin{align*}
c_{i,1}(t,x,y) = \displaystyle \sum_{k=1}^n \partial_{x_k} c_{i,0}(t,x) w_k(y).
\end{align*}
Choosing $\psi_1 = 0$ in  $\eqref{macro_6}$, we obtain with a standard calculation
\begin{eqnarray*}
&&\int_{0}^{T} \int_{\Omega} A_{hom} \nabla \phi_0 \nabla \psi_0 \,dx \,dt = \int_{0}^{T} \int_{\Omega} \sum_{i=1}^{P} z_i c_{i,0} \psi_0 \,dx \,dt\\
&&+\frac{1}{\left | Y^f \right | } \int_{0}^{T} \int_{\Omega} \int_{\Gamma} \xi_1 (x,y) \psi_0 (t,x) \,dS(y) \,dx \,dt + \frac{1}{\left | Y^f \right | }\int_{0}^{T} \int_{\partial \Omega} \xi_2 (x) \psi_0 (t,x) \,dS(x) \,dt.  
\end{eqnarray*}
This is the  weak formulation of $(\ref{Macro_eq_2})$. With similar arguments we obtain from $\eqref{macro_4}$ with $\psi_1 = 0$ the weak formulation for $\eqref{Macro_eq_1}$. This completes the proof of (a).  
\\

The proof of (b) follows easily by noting that the last term of the left hand side of (\ref{macro_1}) goes to $0$ in the limit as $\epsilon \to 0$, for the case $\alpha < \beta$.
\\
\end{proof}

\begin{remark}\label{remark_generalizations}
Let us mention that the setting of the microscopic model cannot be extended very much (except the addition of a reaction term on the right hand side of the transport equation \eqref{non_dim_NP_eq}) without increasing the effort in the homogenization problem substantially. In fact, a more general setting leads to difficulties to control the solutions of the microscopic system explicitly with respect to the scale parameter $\epsilon$. 
\smallskip

(i) Existence for a model involving a Robin boundary condition (instead of a Neumann boundary condition) for $\phi_\epsilon$ was considered, e.g., in \cite{Both14}. However, for such a boundary condition we are not able to prove $\epsilon$-uniform $L^\infty$-estimates for the concentrations. More precisely, the Robin boundary condition induces in the proof of Theorem \ref{Thm_L_infty_est}  an additional boundary integral of the form 
$$
\int_{\partial \Omega_\epsilon} k W_k \phi_\epsilon \,dS (x),
$$
see formula \eqref{L_infty_10}, and it is not clear how to control this term. 

We also remark that our method cannot be extended in an obvious way to the case when $\phi_\epsilon$ satisfies a Dirichlet boundary condition. More precisely, one problem arises in the proof of the energy estimates from Proposition \ref{Thm_energ_est_1}, where $c_{i,\epsilon}^2$ is used as a test function for Poisson's equation (see formula \eqref{energ_est_4}). In the case of Dirichlet boundary conditions for the potential, $c_{i,\epsilon}^2$ is not an admissible test function in the weak sense and leads to additional boundary terms which we are not able to control uniformly with respect to $\epsilon$.
\smallskip
 
(ii) Considering a time dependent $\xi_\epsilon$ in \eqref{non_dim_bc_2}, sufficiently regular with respect to time and satisfying
\begin{eqnarray}
\frac{d}{dt} \int_{\partial \Omega_\epsilon} \xi_\epsilon (t,x) \,dS (x) = 0 \ \text{for a.e. $t \in (0,T)$}, \label{cond_xi_1}\\
\int_{\Omega_\epsilon} \sum_{i=1}^{P} z_i c^0_{i} (x) \,dx+ \int_{\partial \Omega_\epsilon} \xi_\epsilon (0,x) \,dS(x) = 0 \label{cond_xi_2},
\end{eqnarray}
the existence result given in Proposition \ref{thm_exist} remains valid for all $\epsilon>0$ fixed. Here, the assumptions \eqref{cond_xi_1} and \eqref{cond_xi_2} ensure that the required compatibility condition \eqref{existence_2} is satisfied. 

However, the time derivative $\partial_{t} \xi_\epsilon$ causes difficulties in the estimation of time derivative of the energy functional $V_\epsilon(t)$ in Proposition \ref{Thm: V_est}, which is the basis for the a priori estimates needed for the homogenization limit. Indeed, an additional term 
$$
\epsilon ^ \beta \int_{\partial \Omega_\epsilon} \partial_{t} \xi_{\epsilon} {\phi_\epsilon}dS(x)
$$ 
appears in \eqref{V_est_8}. A possibility to control this term would be by restricting the values of the parameters $\alpha, \beta$ and by considering a suitable scaling on $\partial_{t} \xi_\epsilon$ with respect to $\epsilon^{\alpha}$. 
\smallskip

(iii) If the diffusion coefficients are assumed non-constant, namely $D_i \in C^2 ([0,T] \times \overline{\Omega})$ and there exist positive constants $m,M$ such that $m < D_i (t,x)< M$ for all $(t,x) \in [0,T] \times \overline{\Omega}$, then, again,  the existence result in Proposition \ref{thm_exist} remains valid for all fixed $\epsilon>0$. However, problems arise in the homogenization process, especially in the proof of Theorem \ref{Thm_L_infty_est} giving the uniform $L^\infty$-estimate for the concentrations. More precisely, integrating by parts in the term $I_2$ from formula \eqref{L_infty_6}
and using Poisson's equation for $\phi_\epsilon$, we obtain an additional term 
$$
\epsilon ^ {-\alpha+ \beta} z_i \int_{\Omega_\epsilon} \nabla D_i \epsilon ^ \alpha \nabla \phi_\epsilon kW_k \,dx.
$$  
This term can be estimated by 
$$
C \epsilon ^ {-\alpha + \beta} \epsilon^\alpha \lVert \nabla \phi_\epsilon \rVert_{L^4(\Omega_\epsilon)}\lVert k \rVert_{L^2(\{W(t)>k\})}\lVert W_k \rVert_{L^4 (\Omega_\epsilon)}.
$$ 
Now, estimating $\lVert \nabla \phi_\epsilon \rVert_{L^4(\Omega_\epsilon)}$ via the $H^2$-norm of $\phi_\epsilon$ leads to unfavorable $\epsilon$-dependent constants.
\smallskip

(iv) Concerning the reaction terms in \eqref{non_dim_NP_eq}, we emphasize that in \cite{Both14} such terms were considered and regular solutions were obtained for the approximate problem involving the nonlinear diffusion. Using similar assumptions on reaction terms and adding a condition of the form
$$
\sum_{i=1}^{P}  z_i f_i (t,x,s)  = 0 \ \ \forall (t,x,s) \in [0,T] \times \overline{\Omega} \times \mathbb{R}^P
$$
in order to guarantee the compatibility condition \eqref{existence_2} needed to cope with the Neumann problem for the electric potential, we can show that the solutions of the microscopic model have the same regularity properties as those of the approximate system in \cite{Both14}. Furthermore, these reaction terms preserve the a priori estimates and we are able to pass to the limit also in these terms. However, since homogenization processes for semilinear problems are by now rather standard, we have chosen to focus on the nonlinearities arising in the higher order terms like the nonlinear diffusion and the drift term.
\end{remark}


\section{Discussion and outlook}
\label{Discussion_and_outlook}

We performed a rigorous homogenization of a drift-diffusion model for multiple species with nonlinear (non-degenerate) diffusion terms involving the nonlinear function $h_p(r) = r + \eta r^p$, with $\eta >0$ and $p\in [4, \infty)$. This nonlinear problem raises  difficulties for the homogenization procedure. In particular, for getting uniform estimates of the solutions with respect to the scale parameter $\epsilon$, it is necessary to exploit the structure of the system, which admits a nonnegative energy functional  decreasing in time along solutions of the model. Combining such arguments with more classical energy estimates, uniform \textit{a priori} estimates of the microscopic solutions could be established, and a homogenized model could be derived for different scalings of the microscopic problem. 

The microscopic model as well as the homogenized (macroscopic) model both depend on the parameter $\eta>0$. In \cite{Both14} it has been shown that fixing all other parameters of the model, in the limit $\eta \to 0$, the solutions of the nonlinear drift-diffusion model \eqref{non_dim_PNP}-\eqref{def_h_p} converge to a solution of a Poisson-Nernst-Planck (PNP) system (formally obtained by setting $\eta =0$), thus leading to the following diagram: 
\begin{equation*}
\begin{xy}
  \xymatrix{
      u_{\epsilon, \eta} \quad \ar[r]^{ \eta \to 0, \ \cite{Both14} } \ar[d]_{\epsilon \to 0}    &   \quad u_{\epsilon, 0} \ar[d]^{\epsilon \to 0 \, (?)}   \\
      u_{0,\eta} \quad \ar[r]_{\quad \quad \eta \to 0 \,  (?) \quad \quad}              &   \quad u_{0,0}  
  }
\end{xy}
\end{equation*}
Here, $u_{\epsilon, \eta}$ represents the solution of the nonlinear drift-diffusion system \eqref{non_dim_PNP}-\eqref{def_h_p}, $u_{\epsilon, 0}$ represents the solution of a PNP system, and $u_{0, \eta}$ represents the solution to the homogenized model for fixed parameter $\eta$ derived in Theorem \ref{Thm_macro_model}. If we consider the case $\alpha = \beta$, the homogenized system for $u_{0, \eta}$ is again of the form \eqref{non_dim_PNP}-\eqref{def_h_p}, however with matrix-valued constant coefficients. Thus, by similar arguments like in \cite{Both14}, we may assume that we can pass to the limit $\eta \to 0$ to obtain a PNP model with solution $u_{0,0}$. The question which now arises is whether this limit problem can also be obtained by passing to the limit $\epsilon \to 0$ in the PNP-problem for $u_{\epsilon,0}$ (see diagram). Unfortunately, the convergence for $\eta \to 0$ takes place in function spaces with weaker regularity, and thus, the uniform (with respect to $\epsilon$) estimates of the solutions $u_{\epsilon, \eta}$ cannot be transferred to the limit $u_{\epsilon, 0}$ and cannot be used for passing to the limit in the microscopic PNP model. An alternative approach could be to show uniform error estimates (with respect to suitable norms) for the other three convergences in the diagram, which is a highly demanding aim to be addressed in future investigations.


\begin{appendix}\label{appendix}
\section{Auxiliary results}
\label{Preparatory_results}
In this section we prove some auxiliary results which have been used to obtain the uniform estimates for the microscopic solutions  $(c_{i,\epsilon}, \phi_\epsilon)$ in Section \ref{Sect_unif_est} and for the derivation of the macroscopic model in Section \ref{Deriv_macro_model}.  We start with an approximation result.
\begin{lemma}\label{lemma_approx}
Let $ \Omega$ be a bounded open set in $\R^n$ and let $ u \in L^2(0,T; H^1(\Omega)) \cap L^\infty ((0,T) \times \Omega ), u \geq 0$ almost everywhere in $(0,T) \times \Omega$  with $\partial_{t} u \in L^2(0,T;H^1(\Omega)^\prime)$. For each fixed $\delta >0$, there exists a sequence $\psi_m$
in $C^\infty ([0,T] \times \overline{\Omega})$ 
which satisfies  for all $m$
\begin{align*}
 	\psi_m (t,x) \in \left[- \frac{\delta}{2}, \lVert u \rVert_{L^\infty ((0,T) \times \Omega)} + \frac{\delta}{2}\right] \quad \mbox{ for all } \  (t,x) \in   [0,T] \times \overline{\Omega},
\end{align*}
and such that
\begin{enumerate}
[label = (\roman*)]
\item\label{ConvergencePsimL2H1} $\psi_m$ converges to $u$ in $L^2(0,T;H^1(\Omega))$,
\item\label{ConvergencedtPsim}  for every sequence of functions $\upsilon_m$ bounded in $L^2((0,T),H^1(\Omega))$ and $L^{\infty}((0,T)\times \Omega)$ with $\upsilon_m \rightarrow \upsilon$, $\nabla \upsilon_m \rightarrow \nabla \upsilon$ almost everywhere in $(0,T)\times \Omega$ for some $\upsilon \in L^2((0,T),H^1(\Omega))$, it holds that
\begin{align*}
\lim_{m \to \infty} \int_0^T \langle \partial_t \psi_m , \upsilon_m \rangle_{H^1(\Omega)',H^1(\Omega)}dt = \int_0^T \langle \partial_t u, \upsilon\rangle_{H^1(\Omega)',H^1(\Omega)}dt.
\end{align*}
\end{enumerate}
 \end{lemma}
\begin{proof}
First, we observe from \cite[Proposition 23.23, (iii)]{Zei90}, there exists a sequence $\phi_m$ in $C^\infty([0,T] \times \overline{\Omega})$ such that
\begin{eqnarray} \label{p_m_phi1}
\phi_m \rightarrow u \ \text{in $L^2(0,T;H^1(\Omega))$ and} \ \partial_{t}\phi_m \rightarrow \partial_{t}u \ \text{in $L^2(0,T;H^1(\Omega)^\prime)$.}\label{p_m_phi_2}
\end{eqnarray} 
Now, let $\delta >0$ be any fixed number, and let $\rho_\delta \in C^\infty (\R)$ be such that $\rho_\delta, {\rho_\delta}^\prime, {\rho_\delta}^{\prime \prime}$ are bounded functions in $\R$ with $\rho_\delta (x) \in \left[-\frac{\delta}{2},\lVert u \rVert_{L^\infty ((0,T) \times \Omega)} + \frac{\delta}{2}\right]$ for all $x \in \R$ and 
 \begin{equation*}
 \rho_\delta (x) = x,\ \text{if $ x \in \left[-\frac{\delta}{4},\lVert u \rVert_{L^\infty ((0,T) \times \Omega)} + \frac{\delta}{4}\right]$}.
 \end{equation*}
 For example, such a function can be obtained by mollifying the following function:
 \begin{flalign*}
 \tilde{\rho}_\delta(x)=
 \begin{cases}
  \lVert u \rVert_{L^\infty ((0,T) \times \Omega)} + \frac{\delta}{2} &\text{ if $x > \lVert u \rVert_{L^\infty ((0,T) \times \Omega)} + \frac{\delta}{2}$,} \\
 x  &\text{ if $x\in \left[-\frac{\delta}{2},\lVert u \rVert_{L^\infty ((0,T) \times \Omega)} + \frac{\delta}{2}\right]$,} \\
  -\frac{\delta}{2} &\text{ if $x < -\frac{\delta}{2}$}.
 \end{cases}
 \end{flalign*}
 For the sake of clarity, in the remaining part of the proof we skip the index $\delta$ and denote $\rho_\delta$ simply by $\rho$. Note that $\rho (u) = u, \ \rho^\prime ( u) =1, \  \rho^{\prime \prime} (u) = 0$ for almost every $(t,x) \in (0,T) \times \Omega$. Let $\lVert \rho ^\prime \rVert_{L^\infty(\R)}<M,$ for some $M>0$. Since up to a subsequence, which is still indexed by $m$, $\phi_m$ converges to $u$ almost everywhere in $(0,T) \times \Omega$ and $\rho$ is continuous, we have
 \begin{equation*}
 \rho  (\phi_m) \rightarrow \rho (u) = u \ \ \text{a.e. in $(0,T) \times \Omega$.}
 \end{equation*}
 Again $ | \rho (\phi_m )| < C$ in $[0,T] \times \overline{\Omega}$, where $C>0$ is some constant independent of $m$. 
 So by the dominated convergence theorem, we have
 \begin{equation}\label{lemm_approx_1}
\rho (\phi_m) \rightarrow u \ \ \text{in $L^2 (0,T;L^2(\Omega))$.}
 \end{equation}
Again, for $ j \in \{1,...,n\}$, up to a subsequence, still indexed by $m$, we have that
\begin{equation*}
\frac{\partial}{\partial x_j} \rho ( \phi_m)= \rho ^\prime (\phi_m)  \frac{\partial \phi_m}{\partial x_j}
\end{equation*}
converges almost everywhere in $(0,T) \times \Omega$ to
\begin{equation*}
\rho^\prime (u) \frac{\partial u}{\partial x_j} =\frac{\partial u}{\partial x_j}.
\end{equation*}
Furthermore, it holds
\begin{equation*}
\left|\frac{\partial}{\partial x_j} \rho ( \phi_m) \right|= \left | \rho ^\prime (\phi_m)  \frac{\partial \phi_m}{\partial x_j}  \right | \leq M \left |\frac{\partial \phi_m}{\partial x_j} \right | \ \text{a.e. in $(0,T) \times \Omega$.}
\end{equation*}
Thus, by the generalized dominated convergence theorem (see \cite[Exercises 20, 21, p. 59]{Fol99}), we get 
\begin{equation} \label{lemm_approx_2}
\frac{\partial}{\partial x_j}\rho(\phi_m) \rightarrow \frac{\partial u}{\partial x_j}  \ \ \text{in $L^2(0,T;L^2(\Omega))$.}
\end{equation}
Finally, we choose a subsequence of $\rho(\phi_m)$, still indexed by $m$, along which the above convergence results hold and set $\psi_m:= \rho(\phi_m)$. This proves \ref{ConvergencePsimL2H1}. To prove \ref{ConvergencedtPsim} we use
\begin{align}\label{AuxiliaryEquationPhim}
\int_0^T \langle \partial_t \psi_m , \upsilon_m \rangle_{H^1(\Omega)',H^1(\Omega)}dt =  \int_0^T \langle \partial_t \phi_m  , \rho'(\phi_m)\upsilon_m \rangle_{H^1(\Omega)',H^1(\Omega)}dt.
\end{align}
Due to the strong convergence of $\partial_t \phi_m$ in $L^2((0,T),H^1(\Omega)')$, the claim follows if we show the weak convergence of $\rho'(\phi_m) \upsilon_m$ in $L^2((0,T),H^1(\Omega))$. From the assumptions on $\upsilon_m$ and similar arguments as above we get
\begin{align*}
\Vert \rho'(\phi_m)\upsilon_m\Vert_{L^2((0,T),H^1(\Omega))} \le C,
\end{align*}
and 
\begin{align*}
\rho^\prime(\phi_m)\upsilon_m \rightarrow \upsilon, \quad \nabla (\rho^\prime(\phi_m)\upsilon_m) \rightarrow \nabla \upsilon \qquad \mbox{ a.e. in } (0,T)\times \Omega.
\end{align*}
The weak convergence of $\rho^\prime(\phi_m)\upsilon_m$ to $\upsilon$ now follows from \cite[(13.44) Theorem]{HewittStromberg}.
\end{proof}
\begin{lemma}\label{lemma_weak_time_derivative}
Let $c_{i,\epsilon}$ be a weak solution to the microscopic model (\ref{Exist_1})-(\ref{Exist_2}). Then for each $\delta >0$, the following equality holds in the space $L^1(0,T;\R)$:
	\begin{equation*}
	\frac{d}{dt} \int_{\Omega_\epsilon} (c_{i,\epsilon} + \delta) \log(c_{i,\epsilon}+\delta) \,dx = \left< \partial_{t} c_{i,\epsilon}, \log (c_{i,\epsilon}+\delta) \right>_{H^1(\Omega_\epsilon)^\prime,H^1(\Omega_\epsilon)}.
	\end{equation*}

\end{lemma}
\begin{proof}
In order to prove the lemma, we show that the followings hold:
	\begin{itemize}
\item[(i)] \text{$\int_{\Omega_\epsilon} (c_{i,\epsilon} + \delta) \log(c_{i,\epsilon}+\delta) \,dx \in L^1(0,T)$}.
\item[(ii)] \text{$\left< \partial_{t} c_{i,\epsilon}, \log (c_{i,\epsilon}+\delta) \right>_{H^1(\Omega_\epsilon)^\prime,H^1(\Omega_\epsilon)} \in L^1(0,T)$}.
\item[(iii)]For all $\psi \in C^\infty_0(0,T)$, we have 
	\begin{eqnarray*}
	&&\int_{0}^{T}\int_{\Omega_\epsilon} (c_{i,\epsilon} + \delta) \log(c_{i,\epsilon}+\delta) \psi^\prime(t) \,dx \,dt \\
	&&=- \int_{0}^{T}\left< \partial_{t} c_{i,\epsilon}, \log (c_{i,\epsilon}+\delta) \right>_{H^1(\Omega_\epsilon)^\prime,H^1(\Omega_\epsilon)} \psi(t) \,dt.
	\end{eqnarray*}
\end{itemize}
Statement (i) follows immediately from the fact that $c_{i,\epsilon} \in L^\infty((0,T) \times \Omega_\epsilon)$ and is non-negative. (ii) is also obtained easily by noting that
	\begin{eqnarray*}
	&&\int_{0}^{T}\left|\left< \partial_{t} c_{i,\epsilon}, \log (c_{i,\epsilon}+\delta) \right>_{H^1(\Omega_\epsilon)^\prime,H^1(\Omega_\epsilon)} \right|\,dt \\
	&& \leq \int_{0}^{T} \lVert \partial _{t}c_{i,\epsilon}\rVert_{H^1(\Omega_\epsilon)^\prime} \lVert \log (c_{i,\epsilon}+\delta)\rVert_{H^1(\Omega_\epsilon)}\,dt\\
	&&\leq  \lVert \partial _{t}c_{i,\epsilon}\rVert_{L^2(0,T;H^1(\Omega_\epsilon)^\prime)} \lVert \log  (c_{i,\epsilon}+\delta)\rVert_{L^2(0,T;H^1(\Omega_\epsilon))} < \infty.
	\end{eqnarray*}
Next, let us prove (iii).
Lemma \ref{lemma_approx} guarantees the existence of a sequence $\psi_m$ in $C^\infty ([0,T] \times \overline{\Omega}_\epsilon )$ converging strongly to $c_{i,\epsilon}$ in $L^2((0,T),H^1(\Omega_{\epsilon}))$ and for all $m$, the range of $\psi_m \subseteq  [- \frac{\delta}{2}, \lVert c_{i,\epsilon} \rVert_{L^\infty ((0,T) \times \Omega_\epsilon)} + \frac{\delta}{2}]$. Then, 
	   \begin{eqnarray}\label{lemm_time_1}
	   &&\int_{0}^{T}\int_{\Omega_\epsilon} (c_{i,\epsilon} + \delta) \log(c_{i,\epsilon}+\delta) \psi^\prime(t) \,dx \,dt \nonumber \\
	   &&= \lim_{m \to \infty } \int_{0}^{T}\int_{\Omega_\epsilon} [\psi_m + \delta] \log[\psi_m +\delta] \psi^\prime(t) \,dx \,dt.
	   \end{eqnarray}
Note that $\log[\psi_m+\delta]$ is well-defined, since $\psi_m \geq -\frac{\delta}{2}$. After an integration by parts in time, the right hand side of (\ref{lemm_time_1}) becomes
	   \begin{eqnarray}\label{lemm_time_2}
	  &&- \lim_{m \to \infty} \int_{0}^{T} \int_{\Omega_\epsilon} \partial_{t} \psi_m  \log[\psi_m+\delta]\psi(t)\,dx \,dt \nonumber \\
	  &&-\lim_{m\to \infty} \int_{0}^{T} \int_{\Omega_\epsilon} \partial_{t} \psi_m  \psi(t) \,dx \,dt =: -I - II. \nonumber 
	   \end{eqnarray}
Choosing $\upsilon_m := \log(\psi_m + \delta) \psi$ in Lemma \ref{lemma_approx} \ref{ConvergencedtPsim}, it is easy to check that by the convergence of $\psi_m $ in $L^2((0,T),H^1(\Omega))$, the sequence $\upsilon_m$ fulfills the assumptions of Lemma \ref{lemma_approx} with $\upsilon = \log(c_{i,\epsilon} +\delta)\psi$. Hence, we obtain
\begin{align*}
I = \int_{0}^{T}\left<\partial_{t} c_{i,\epsilon}, \log(c_{i,\epsilon}+\delta)\right>_{H^1(\Omega_\epsilon)^\prime,H^1(\Omega_\epsilon)} \psi(t) \,dt.
\end{align*} 	   
In the same way we obtain
\begin{align*}
II =  \int_{0}^{T} \left < \partial_{t} c_{i,\epsilon}, 1\right>_{H^1(\Omega_\epsilon)^\prime,H^1(\Omega_\epsilon)} \psi(t)\,dt=0,
\end{align*}
where the last equality follows from testing the equation \eqref{Exist_1} with $1$.	  
\end{proof}
 
\begin{lemma}\label{lemma_extension}
For $1 \leq q < \infty$ there exists an extension operator \,  $\widetilde{\cdot}: W^{1,q}(\Omega_\epsilon) \to W^{1,q}(\Omega)$ such that for all $u_\epsilon \in W^{1,q}(\Omega_\epsilon)$, we have
\begin{equation}\label{lemm_extension}
     \lVert \widetilde{u}_\epsilon \rVert_{L^q(\Omega)} \leq C \lVert u_\epsilon \rVert_{L^q(\Omega_\epsilon)} , \ 
     \lVert \nabla \widetilde{u}_\epsilon \rVert_{L^q(\Omega)} \leq C \lVert \nabla u_\epsilon \rVert_{L^q(\Omega_\epsilon)},
\end{equation}
with a constant $C >0$ independent of $\epsilon$. 

Furthermore, if  $u_\epsilon \in W^{1,q}(\Omega_\epsilon) \cap L^\infty (\Omega_\epsilon)$, then $\widetilde{u}_\epsilon \in L^\infty (\Omega)$ and 
$$
\lVert \widetilde{u}_\epsilon \rVert _{L^\infty (\Omega)} \leq C \lVert u_\epsilon \rVert_{L^\infty(\Omega_\epsilon)},
$$ with a constant $C>0$ independent $\epsilon$, and if $u_\epsilon \in W^{1,q}(\Omega_\epsilon)$ is a non-negative function, so is the extension $\widetilde{u}_\epsilon$. 
\end{lemma}
\begin{proof}
The existence of an extension operator satisfying the estimates \eqref{lemm_extension} is by now a standard result in homogenization theory, see, e.g., \cite{Cio79} and \cite{Acerbi1992}, and is based on a similar result in the standard cell followed by a decomposition of the domain $\Omega_\epsilon$ in $\epsilon$-cells and a scaling argument. Since however, we need further properties of the extension operator (like non-negativity and essential boundedness), let us sketch the construction of the extension operator in the standard cell, from which the additional properties can be derived.  
\\
In a first step we extend a function $u \in W^{1,q}(Y^f)$ by a reflection method to the whole cell $Y$. More precisely, we define the extension operator $P: W^{1,q} (Y^f) \to W^{1,q} (Y)$ in the following way (see \cite[Theorem 9.7]{Bre11} for more details and \cite{Nec12} for more general settings and higher order derivatives):  let $U_1,...,U_k$ be an open covering of $\Gamma$ such that $U_i \subset Y$ and assume $\theta_i \in C_0^\infty (U_i)$ for all $i \leq 1 \leq k$; $\theta_0 \in C^\infty (\overline{Y})$, supp $\theta_0 \subset \overline{Y} \setminus \Gamma$, with $\sum_{i=0}^{k} \theta _i =1$ and $0 \le \theta_i \le 1$ for $i=0,\ldots,k$. We define
\begin{align}\label{LocalReflectionExtension}
\bar{u}_0(y):= 
\begin{cases}
\theta_0 (y) u (y) &\text{if $y\in Y^f$,} \\
0 &\text{ if $y\in Y \setminus Y^f$;}
\end{cases}
\qquad 
\hat{u}_i(y):= 
\begin{cases}
\theta_i (y) w_i (y) &\text{if $y\in U_i$,} \\
0 &\text{if $y\in Y \setminus U_i$,}
\end{cases}
\end{align}
for $1 \leq i \leq k$, where the functions $w_i$ represent the extension of $\left.u\right|_{U_i \cap Y^f}$ to $U_i$ by reflection. We emphasize that 
\begin{align}\label{AuxiliaryInequalityLinftyExtension}
\Vert w_i \Vert_{L^{\infty}(U_i)} \le \Vert  u \Vert_{L^{\infty}(U_i \cap Y^f)}.
\end{align}
 Then we define the extension operator
\begin{equation*}
Pu = \bar {u}_0 + \sum_{i=1}^{k} \hat{u}_i.
\end{equation*}
From $\eqref{LocalReflectionExtension}$ and $\eqref{AuxiliaryInequalityLinftyExtension}$ we immediately obtain for $u \in W^{1,q}(Y^f)\cap L^{\infty}(Y^f)$ and non-negative that
\begin{align}\label{EssentialLowUpBoundExtension}
0 \le Pu(y) \le \Vert u \Vert_{L^{\infty}(Y^f)} \qquad \mbox{f.a.e. } y \in Y.
\end{align}
The crucial point in the construction of the global extension operator on $\Omega_{\epsilon}$ is that the norm of the gradient of the extension can be estimated by norm of the gradient of the function itself. However, in general this is not the case for the operator $P$. Hence, in the second step, we construct a local extension operator $S: W^{1,q}(Y^f)\rightarrow W^{1,q}(Y)$ such that
\begin{align}\label{AuxiliaryInequalitiyExtensionGradient}
\Vert \nabla Su \Vert_{L^q(Y)} \le C \Vert \nabla u\Vert_{L^q(Y^f)}.
\end{align}
Since $Y^f$ is connected, the existence of such an operator is well-known, see 
\cite{Cio79}  or \cite{Acerbi1992}. We only have to show that the operator $S$ also preserves the non-negativity and essential boundedness of a function. We denote the mean-value on $Y^f$ of a function $u$ by $(u)_{Y^f}$. Now, we define as in \cite[Lemma 3]{Cio79} the extension operator
\begin{align*}
Su := P\left(u - (u)_{Y^f}\right) + (u)_{Y^f},
\end{align*}
which especially fulfills $\eqref{AuxiliaryInequalitiyExtensionGradient}$. Further, for $u \in W^{1,q}(Y^f)\cap L^{\infty}(Y^f)$ non-negative, we obtain from $\eqref{EssentialLowUpBoundExtension}$ 
\begin{align*}
\Vert Su \Vert_{L^{\infty}(Y)} \le \Vert u\Vert_{L^{\infty}(Y^f)} + 2\vert (u)_{Y^f}\vert \le C \Vert u \Vert_{L^{\infty}(Y^f)},
\end{align*}
and
\begin{align*}
Su = Pu \underbrace{- P(u)_{Y^f}}_{\geq - (u)_{Y^f}} + (u)_{Y^f} \geq 0.
\end{align*}
This completes the proof.
\end{proof}

\begin{lemma}\label{lemma_mean_value}
 For all $u_\epsilon \in H^1(\Omega_\epsilon)$ with $\frac{1}{|\Omega_\epsilon|} \int_{\Omega_\epsilon} u_\epsilon \,dx =0$, we have
 \begin{equation}\label{Lemma_mean_value}
  \lVert u_\epsilon \rVert _{L^2(\Omega_\epsilon)} \leq C \lVert \nabla u_\epsilon \rVert_{L^2(\Omega_\epsilon)},
 \end{equation}
with a constant $C>0$ independent of $\epsilon$.
\end{lemma}
\begin{proof} The proof is similar to the proof of \cite[Lemma 1.3]{Mik00}. Since it is rather short,  we include it here for the sake of completeness. Let $\widetilde{u}_\epsilon$ be the extension of $u_\epsilon$ to $\Omega$ given in Lemma \ref{lemma_extension}. Now, we use the Poincaré–Wirtinger inequality and the fact that due to the zero mean value of $u_\epsilon$ on $\Omega_\epsilon$ we have 
$
 \int_{\Omega} \widetilde{u}_\epsilon =  \int_{\Omega \setminus \Omega_\epsilon} \widetilde{u}_\epsilon.
$
We obtain
 \begin{eqnarray*}
 \lVert \widetilde{u}_\epsilon \rVert _{L^2(\Omega)} &\leq& \left \lVert \widetilde{u}_\epsilon  - \frac{1}{|\Omega|} \int_{\Omega} \widetilde{u}_\epsilon  \right \rVert_{L^2(\Omega)}+ \frac{1}{|\Omega|} \left | \int_{\Omega} \widetilde{u}_\epsilon  \right | | \Omega|^{\frac{1}{2}}\\
 &  \leq &C (\Omega) \lVert \nabla \widetilde{u}_\epsilon \rVert_{L^2(\Omega)}+ \frac{1}{|\Omega|^{\frac{1}{2}}} \int_{\Omega\setminus \Omega_\epsilon} |\widetilde{u}_\epsilon | \,dx \\
 &  \leq&  C(\Omega) \lVert \nabla \widetilde{u}_\epsilon  \rVert_{L^2(\Omega)} + \frac{|\Omega \setminus \Omega_\epsilon|^{\frac{1}{2}}}{|\Omega|^{\frac{1}{2}}} \lVert \widetilde{u}_\epsilon  \rVert_{L^2(\Omega \setminus \Omega_\epsilon)} \\
 & \leq & C(\Omega) \lVert \nabla \widetilde{u}_\epsilon  \rVert_{L^2(\Omega)} + \frac{|\Omega \setminus \Omega_\epsilon|^{\frac{1}{2}}}{|\Omega|^{\frac{1}{2}}} \lVert \widetilde{u}_\epsilon  \rVert_{L^2(\Omega)}.
 \end{eqnarray*}
Since $\frac{|\Omega \setminus \Omega_\epsilon|^{\frac{1}{2}}}{|\Omega|^{\frac{1}{2}}} < C_1 <1$, with $C_1$ is independent of $\epsilon$, we have
\begin{eqnarray*}
 (1-C_1) \lVert \widetilde{u}_\epsilon  \rVert _{L^2(\Omega)} \leq C(\Omega) \lVert \nabla \widetilde{u}_\epsilon  \rVert_{L^2(\Omega)}.
\end{eqnarray*}
Now using the properties (\ref{lemm_extension}) of the extension $\widetilde{u}_\epsilon$, we conclude the proof.
\end{proof}

The following Lemma gives a trace-estimate with an explicit dependence of $\epsilon$ and is well-known in the theory of homogenization, so we skip the proof, which is based on a standard decomposition argument and the trace-inequality in the reference element.
\begin{lemma}\label{TraceInequalityLemma}
For every $u_{\epsilon} \in W^{1,q}(\Omega_{\epsilon})$ with $q \in [1,\infty)$ it holds that
\begin{align*}
\epsilon\Vert u_{\epsilon} \Vert_{L^q(\Gamma_{\epsilon})}^q  \le  C\left( \Vert u_{\epsilon}\Vert_{L^q(\Omega_{\epsilon})}^q + \epsilon^q  \Vert \nabla u_{\epsilon} \Vert^q_{L^q(\Omega_{\epsilon})}\right),
\end{align*}
for a constant $C >  0$ independent of $\epsilon$.
\end{lemma}
\end{appendix}


\section*{Acknowledgements}
AB acknowledges the support by the RTG 2339 ``Interfaces, Complex Structures, and Singular Limits'' of the German Science Foundation (DFG). MG was supported by the SCIDATOS project, funded by the Klaus Tschira Foundation (grant 00.0277.2015). The authors thank the anonymous referees for valuable comments and suggestions which improved the manuscript. In particular, Remark \ref{remark_generalizations} was added to the manuscript to discuss further extensions of the model.


\end{document}